\PassOptionsToPackage{table,xcdraw}{xcolor}
\documentclass[onefignum,onetabnum]{siamonline220329}

\usepackage[table]{xcolor}
\usepackage[scaled=.98]{BOONDOX-uprscr}
\usepackage[utf8]{inputenc}
\usepackage{yfonts}
\usepackage{eufrak}
\usepackage[english]{babel}
\usepackage{graphicx}
\usepackage{fancyhdr}
\usepackage{amsfonts,amssymb,amsmath,mathtools,enumitem}
\usepackage[percent]{overpic}
\usepackage{tikz,mathrsfs,xargs}
\usepackage[colorinlistoftodos,prependcaption,textsize=tiny]{todonotes}
\newcommandx{\at}[2][1=]{\todo[linecolor=red,backgroundcolor=red!25,bordercolor=red,#1]{#2}}

\DeclareTextFontCommand{\textgoth}{\mathfrak}
\DeclareMathAlphabet{\mathpzc}{OT1}{pzc}{m}{it}

\usepackage{hyperref}
\usepackage{booktabs}
\usepackage{varwidth}
\usepackage{cleveref}
\usepackage{cleveref}
\usepackage{algorithm,float,refcount}
\usepackage{algpseudocode}

\definecolor{rev1}{HTML}{cb270f}
\definecolor{rev2}{HTML}{1c8235}
\definecolor{rev3}{HTML}{3815B9}

\usepackage{ifthen,tabularx,graphicx,multirow,refcount}

\usepackage{rotating}
\usepackage{etoolbox}
\usepackage[most]{tcolorbox}
\usepackage{bbm}

\usepackage{enumitem}
\setlist[enumerate]{leftmargin=.5in}
\setlist[itemize]{leftmargin=.5in}

\newcommand{\Fv}{\mathbf{F}}
\newcommand{\Gv}{\mathbf{G}}

\newcommand{\Iv}{\mathbf{I}}

\newcommand{\Kv}{\mathbf{K}}

\newcommand{\Qv}{\mathbf{Q}}
\newcommand{\Rv}{\mathbf{R}}

\newcommand{\Uv}{\mathbf{U}}
\newcommand{\Vv}{\mathbf{V}}
\newcommand{\Wv}{\mathbf{W}}

\newcommand{\Psiv}{\mathbf{\Psi}}

\newcommand{\cv}{\mathbf{c}}

\newcommand{\gv}{\mathbf{g}}
\newcommand{\hv}{\mathbf{h}}

\newcommand{\kv}{\mathbf{k}}

\newcommand{\xv}{\mathbf{x}}
\newcommand{\yv}{\mathbf{y}}

\newcommand*{\dd}{{\,\mathrm{d}}}

\newcommand*{\spec}{{\mathrm{Sp}}}

\newsiamremark{remark}{Remark}
\newsiamremark{example}{Example}
\newsiamremark{hypothesis}{Hypothesis}
\crefname{hypothesis}{Hypothesis}{Hypotheses}
\newsiamthm{claim}{Claim}

\headers{Rigged Dynamic Mode Decomposition}{M. J. Colbrook, C. Drysdale, A. Horning}

\title{Rigged Dynamic Mode Decomposition:\\Data-Driven Generalized Eigenfunction Decompositions for Koopman Operators}

\author{Matthew J. Colbrook\thanks{Department of Applied Mathematics and Theoretical Physics, University of Cambridge, Cambridge, UK (\email{m.colbrook@damtp.cam.ac.uk}).}
  \and Catherine Drysdale\thanks{Centre for Systems Modelling and Quantitative Biomedicine, University of Birmingham, Birmingham, UK (\email{c.n.d.drysdale@bham.ac.uk}).} \and Andrew Horning\thanks{Department of Mathematics, Massachusetts Institute of Technology, Cambridge, USA (\email{horninga@mit.edu}).}}

\newcommand*{\pp}{{[{-}\pi,\pi]_{\mathrm{per}}}}

\newcommand*{\llog}{{\log(1+\epsilon^{-1})}}

\renewcommand{\L}{\mathcal{L}}
\newcommand{\K}{\mathcal{K}}

\newcommand{\kdmd}{\mathbb{K}_{\text{EDMD}}}

\usepackage{bm}

\begin{document}

\maketitle

\begin{abstract}
We introduce the Rigged Dynamic Mode Decomposition (Rigged DMD) algorithm, which computes generalized eigenfunction decompositions of Koopman operators. By considering the evolution of observables, Koopman operators transform complex nonlinear dynamics into a linear framework suitable for spectral analysis. While powerful, traditional Dynamic Mode Decomposition (DMD) techniques often struggle with continuous spectra. Rigged DMD addresses these challenges with a data-driven methodology that approximates the Koopman operator's resolvent and its generalized eigenfunctions using snapshot data from the system's evolution. At its core, Rigged DMD builds wave-packet approximations for generalized Koopman eigenfunctions and modes by integrating Measure-Preserving Extended Dynamic Mode Decomposition with high-order kernels for smoothing. This provides a robust decomposition encompassing both discrete and continuous spectral elements. We derive explicit high-order convergence theorems for generalized eigenfunctions and spectral measures. Additionally, we propose a novel framework for constructing rigged Hilbert spaces using time-delay embedding, significantly extending the algorithm's applicability (Rigged DMD can be used with any rigging). We provide examples, including systems with a Lebesgue spectrum, integrable Hamiltonian systems, the Lorenz system, and a high-Reynolds number lid-driven flow in a two-dimensional square cavity, demonstrating Rigged DMD's convergence, efficiency, and versatility. This work paves the way for future research and applications of decompositions with continuous spectra.
\end{abstract}



\section{Introduction}

Consider the discrete-time dynamical system defined by
\begin{equation}
\label{eq:DynamicalSystem} 
\xv_{n+1} = \Fv(\xv_n), \qquad n= 0,1,2,\ldots.
\end{equation}
Here, $\xv_n\in\Omega$ denotes the state of the system after $n$ iterations from an initial state $\xv_0$, $\Omega\subseteq\mathbb{R}^d$ is the state space, and the function $\Fv:\Omega \rightarrow \Omega$ governs the evolution of the state. In many modern applications, the function $\Fv$ may be unknown since explicit knowledge of the laws governing the dynamical evolution $\xv_0,\xv_1,\xv_2,\ldots,$ may be absent or incomplete. Instead, one observes a collection of so-called ``snapshots" of the system, i.e.,
\begin{equation}\label{eq:snapshots}
\left(\xv^{(m)},\yv^{(m)}\right)\quad \text{such that}\quad \yv^{(m)}=\Fv(\xv^{(m)}),\qquad\text{for}\qquad m=1,\ldots,M.
\end{equation}
These snapshots could arise from a single trajectory or multiple trajectories.
The question is, how can one use this data to study
the dynamical system in~\cref{eq:DynamicalSystem}?

Originally developed in the early $20^{\rm th}$ century \cite{koopman1931hamiltonian,koopman1932dynamical}, Koopman operator theory has recently emerged as a particularly powerful framework for data-driven analysis of dynamical systems. This theory shifts the focus from state-space dynamics to a space of observables $g:\Omega\rightarrow\mathbb{C}$. As the state evolves under the map $\Fv$, the observables evolve under the action of the \textit{Koopman operator} $\mathcal{K}$, defined via composition with $\Fv$ as
\begin{equation}\label{eq:Koopman_action}
[\mathcal{K}g](\xv) = g(\Fv(\xv)), \quad \text{so that}\quad [\mathcal{K}g](\xv_n)=g(\xv_{n+1}).
\end{equation}
In contrast with the state space dynamics, the evolution of the observables under the action of $\mathcal{K}$ is \textit{linear} and can be studied with classical tools based on modal decompositions. Furthermore, finite-dimensional approximations of $\mathcal{K}$ can be constructed from the snapshot data in~\cref{eq:snapshots} using standardized tools from numerical linear algebra and statistics. The combination of interpretable modal decompositions with accessible algorithms for their computation has sparked a surge of interest in data-driven, Koopman-powered algorithms for analysis, model reduction, prediction, and control of dynamical systems (see the reviews~\cite{mezic2013analysis,budivsic2012applied,brunton2021modern,colbrook2023multiverse}).

Although it is linear, the Koopman operator typically operates on an infinite-dimensional space of observables
because, even for a single observable $g$, the space generated by $\mathcal{K}^ng$ is generically infinite-dimensional. This poses two main challenges. Standard finite-dimensional approximations may induce spurious eigenvalues or completely miss regions of spectra~\cite{colbrookthesis,colbrook2022computation,colbrook2022foundations,colbrook2019compute}. Spectra of Koopman operators may also be far more complex than for matrices, with continuous spectra and a lack of nontrivial eigenspaces playing a role in the associated dynamics~\cite{colbrook2021computingCIMP,colbrook2021computing}. Recently developed methods such as Residual Dynamic Mode Decomposition (ResDMD) \cite{colbrook2021rigorous,colbrook2023residualJFM} and Measure-Preserving Extended Dynamic Mode Decomposition (mpEDMD) \cite{colbrook2023mpedmd} successfully address these complexities in a \textit{Hilbert space framework}: they rigorously approximate Koopman spectra, pseudospectra, eigenfunctions, pseudoeigenfunctions, and (for unitary Koopman operators) spectral measures (see also the related works in~\cref{sec:connections_prev_work}). However, the Hilbert space framework does not provide a complete set of coherent modes, complicating modal analysis, model-order reduction, and other tasks when a continuous spectrum is present.

\begin{figure}
\centering
\includegraphics[width=0.9\textwidth]{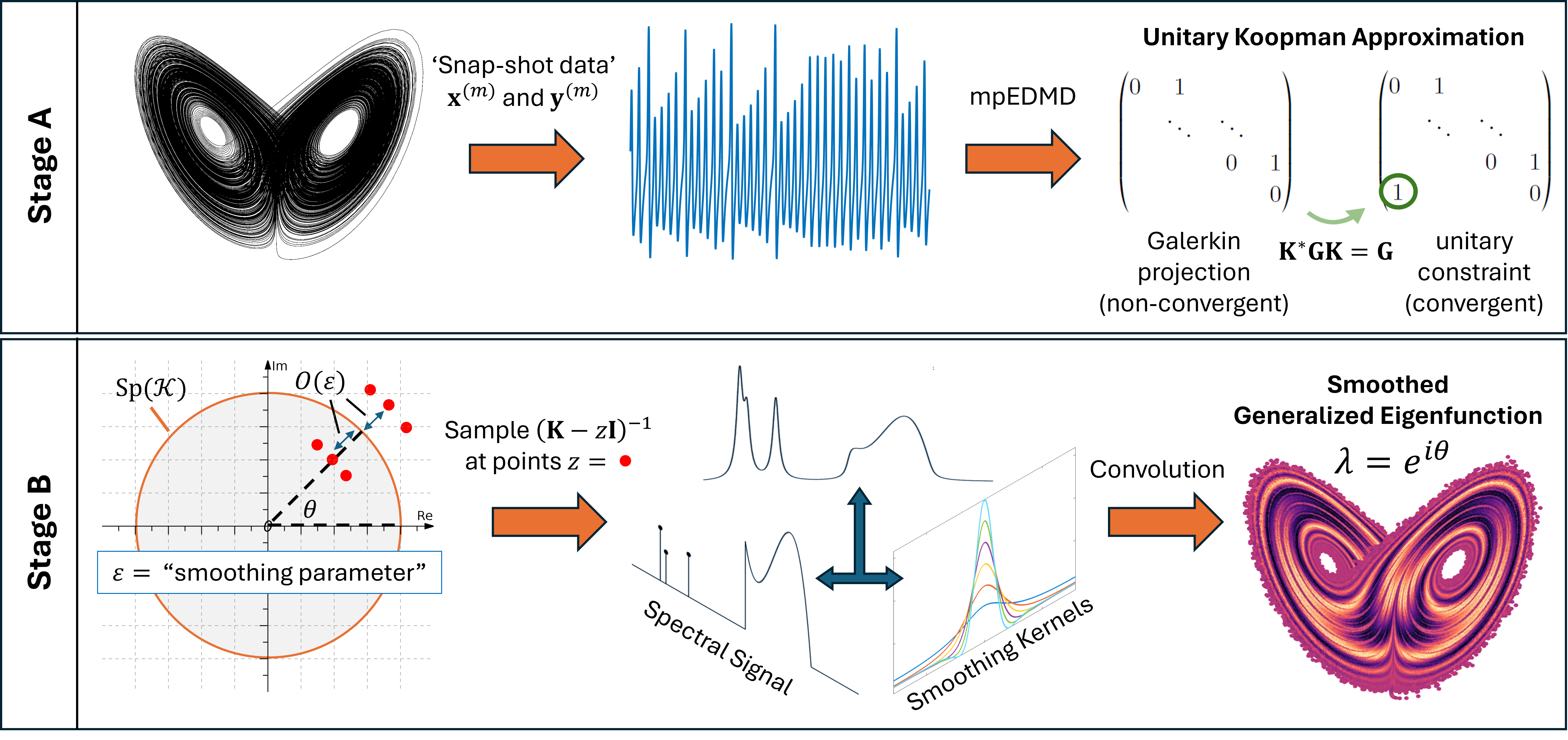}
\caption{Schematic of Rigged DMD. Stage A consists of a data-driven unitary approximation of $\K$. Stage B consists of sampling the resolvent of this approximation to form a smoothed generalized Koopman eigenfunction (and Koopman modes) in the form of a wave-packet approximation. These generalized eigenfunctions provide a decomposition, even in the presence of continuous spectra. The full algorithm is given in \cref{alg:RiggedDMD}.\label{fig:rigged_DMD}}
\end{figure}

\begin{figure}
\centering
\includegraphics[width=0.9\textwidth]{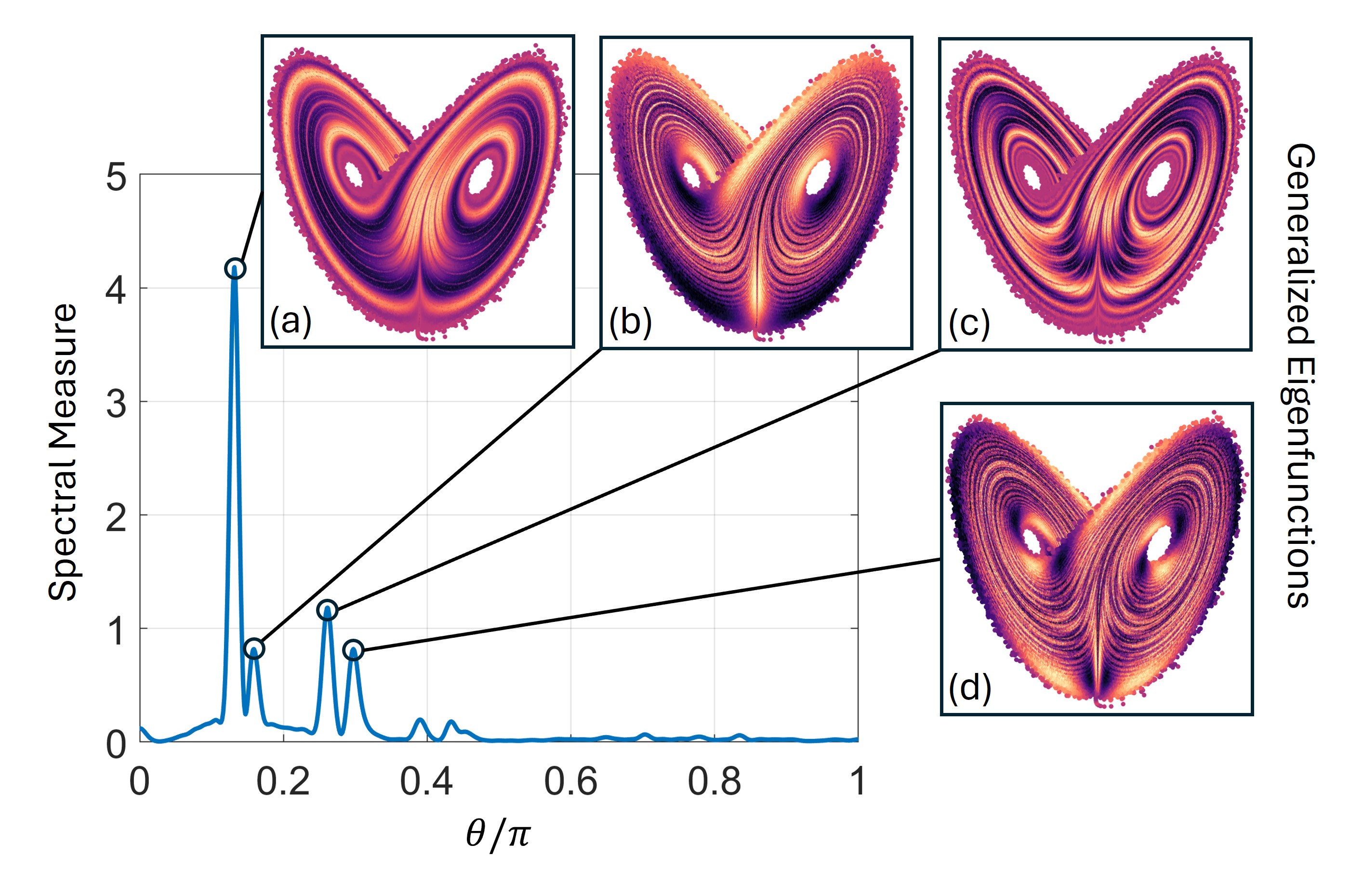}
\caption{Spectral measure and generalized eigenfunctions for the Lorenz system computed using Rigged DMD. The spectral measure is continuous, and the generalized eigenfunctions are coherent.\label{fig:lorenz1}}
\end{figure}

This paper introduces a new Koopman mode decomposition and computational framework for unitary Koopman operators on a \textit{rigged Hilbert space} (see \cref{sec:rigged_H}). In the rigged Hilbert space, the discrete and continuous spectrum are treated on equal footing, and (almost) every point in the spectrum is associated with a set of coherent \textit{generalized eigenfunctions}.\footnote{The terminology is not to be confused with generalized eigenvectors that generate Jordan chains of a finite non-normal matrix \cite{drmavc2023data}.} Generalized eigenfunctions associated with the continuous spectrum are not elements of the Hilbert space. Instead, as distributions in the rigged Hilbert space, they represent coherent linear features of the observables that may be used to identify coherent patterns in~\cref{eq:DynamicalSystem}. 

Our algorithm, which we call \textit{Rigged DMD}, computes smooth \textit{wave-packet approximations} to generalized Koopman eigenfunctions using a two-stage framework (see~\cref{fig:rigged_DMD}). In the first stage, we use mpEDMD to construct a finite-dimensional unitary approximation to the Koopman operator from the snapshot data in~\cref{eq:snapshots}. In the second stage, we use the mpEDMD resolvent to construct wave packets of observables with highly localized spectral content. These approximations converge to generalized eigenfunctions associated with both the discrete and the continuous spectrum (see~\cref{sec:stones_formula}). They are approximately coherent (see~\cref{prop:evc_approx_atom}), achieve high orders of accuracy under natural regularity conditions on the spectral measure of $\mathcal{K}$ (see~\cref{thm:conv_rates}), and converge pointwise in regions where generalized eigenfunctions can be identified with smooth functions (see~\cref{thm:conv_rates2}). For example, \Cref{fig:lorenz1} shows four wave-packet approximations to generalized eigenfunctions associated with the continuous spectrum of a Koopman operator on the Lorenz attractor. Rigged DMD provides the first general procedure for computing generalized eigenfunctions of Koopman operators.

There is extensive literature on rigged Hilbert spaces for self-adjoint operators in PDE analysis and quantum mechanics~\cite{gelfand1955decomp,browder1954eigenfunction,kato1970theory,bohm1989dirac,berezanskii1968expansions}. However, standard rigging techniques do not always apply to Koopman operators (consider the task of constructing a nuclear space on the Lorenz attractor). For this reason, we provide a novel method to construct nuclear spaces for unitary operators with sparse, infinite-dimensional matrix representations (see~\cref{sec:nuclear_construction}). This technique can be applied to \textit{any} Koopman operator via time-delay embedding, enabling the computation of generalized Koopman eigenfunctions for the Lorenz attractor, for example. Moreover, Rigged DMD does not rely on time-delay embedding - it can be applied to any rigging, not necessarily constructed using time-delay embedding.

The paper is organized as follows. We review the spectral theory of unitary Koopman operators on Hilbert spaces and derive Koopman mode decompositions in a rigged Hilbert space in~\cref{sec:background}. The Rigged DMD algorithm is introduced in~\cref{sec:rigged_dmd_alg}, and connections with previous work are discussed in~\cref{sec:connections_prev_work}. Data-driven approximations of the Koopman resolvent are discussed in~\cref{sec:res_comp}, and wave-packet approximations are analyzed in~\cref{sec:stones_formula}. In~\cref{sec:nuclear_construction}, we develop a methodology to rig Hilbert spaces for Koopman operators. A variety of worked examples in~\cref{sec:examples} include: systems with a Lebesgue spectrum, illustrated by Arnold's cat map; integrable systems, illustrated by the nonlinear pendulum; the Lorenz system, demonstrating the utility of the rigged Hilbert space construction via time-delay embedding; and a high-dimensional fluid flow scenario. General purpose code and all the examples of this paper can be found at: \url{https://github.com/MColbrook/Rigged-Dynamic-Mode-Decomposition}.

\section{Background: Koopman mode decompositions}\label{sec:background}

We assume that the system in \cref{eq:DynamicalSystem} preserves a Borel measure $\omega$ (not necessarily finite), i.e., for measurable sets $S\subset\Omega$,
$$
\omega(S)=\omega(\Fv^{-1}(S)),\quad\text{where}\quad \Fv^{-1}(S)=\{\xv:\Fv(\xv)\in S\}.
$$
Such systems preserve an energy or volume associated with $\omega$ and are ubiquitous \cite{arnold1989mathematical,dubrovin2012modern,shields1973theory,hill1986introduction,walters2000introduction}. Many systems admit invariant measures \cite{kryloff1937theorie} or measure-preserving post-transient behavior \cite{mezic2005spectral}. We also assume that the dynamical system is invertible modulo $\omega$-null sets (see \cite[Chapter 7]{eisner2015operator}), so that $\K$ is a unitary operator on the Hilbert space $L^2(\Omega,\omega)$: 
$$
\mathcal{K}:L^2(\Omega,\omega)\rightarrow L^2(\Omega,\omega), \qquad\text{and}\qquad \mathcal{K}\mathcal{K}^*=\mathcal{K}^*\mathcal{K}=I.
$$
Here, $\mathcal{K}^*$ denotes the adjoint of $\mathcal{K}$, and $I$ denotes the identity operator. The spectral theory of unitary operators provides a firm foundation for Koopman mode decompositions.

\subsection{Eigenvalues and eigenfunctions}

Suppose first that an observable
$\phi\in L^2(\Omega,\omega)$ is an eigenfunction of $\K$ with eigenvalue $\lambda$. Then $\phi$ is a coherent observable, because
\begin{equation}
\label{eq:perfectly_coherent}
\phi(\xv_n)=[\mathcal{K}^n\phi](\xv_0)=\lambda^n \phi(\xv_0) \qquad\text{for}\qquad n=0,1,2,\ldots.
\end{equation}
The oscillation and decay/growth of the observable $\phi$ are dictated by the complex argument and absolute value of the eigenvalue $\lambda$, respectively. 
Eigenpairs $(\phi,\lambda)$ encode key information about coherent structures observed in the dynamics in~\cref{eq:DynamicalSystem}.\footnote{However, any constant function is an eigenfunction with eigenvalue one, assuming it is normalizable. This pair is considered trivial because it contains no dynamic information.} For example, the level sets of eigenfunctions determine invariant manifolds \cite{mezic2015applications} and isostables \cite{mauroy2013isostables} of the dynamical system.

If $\mathcal{K}$ has an invariant subspace spanned by eigenfunctions $\{\phi_n\}_{n=1}^\infty$, the \textit{Koopman mode decomposition} (KMD) expands observables $\gv = [g_1,\ldots,g_l]^T$ in this subspace~\cite{mezic2005spectral,rowley2009spectral}:
\begin{equation}\label{eq:classic_KMD}
\gv(\xv) = \sum_{n=1}^\infty \cv_n\phi_n(\xv), \qquad \text{so that}\qquad \gv(\xv_k)=\sum_{n=1}^\infty \cv_n\lambda_n^k\phi_n(\xv_0).
\end{equation}
Here, $\lambda_n$ is the eigenvalue of $\mathcal{K}$ satisfying $\mathcal{K}\phi_n=\lambda_n\phi_n$, for $n=1,2,3,\ldots,$ and the expansion coefficients $\cv_n$ are called the \textit{Koopman modes} of $\gv(\xv)$~\cite{rowley2009spectral}. The KMD provides key insights into the dynamical system, and truncated series may be used for reduced-order models \cite{rowley2010reduced,mohr2022koopman}.

\subsection{Spectrum and pseudoeigenfunctions}

$\K$ may not have non-trivial eigenpairs, e.g., in \cref{fig:lorenz1}. Instead, the correct generalization of the set of eigenvalues of $\mathcal{K}$ is the \textit{spectrum}
$$
\mathrm{Sp}(\mathcal{K})=\left\{z\in\mathbb{C} :(\mathcal{K} - zI)^{-1}\text{ does not exist as a bounded operator}\right\}\subset\mathbb{C}.
$$
Since $\mathcal{K}$ is unitary, we have $\mathrm{Sp}(\mathcal{K})\subset\mathbb{T}=\{z\in\mathbb{C}:|z|=1\}$.
The spectrum includes the set of eigenvalues of $\mathcal{K}$, but in general, it may contain points that are not eigenvalues and hence do not have an associated eigenfunction. However, there are \textit{approximately coherent} observables associated with each point because $\mathcal{K}$ is unitary (therefore, normal). For each point $\lambda\in\mathrm{Sp}(\mathcal{K})$ and any $\delta>0$, there exists a $\phi_\delta\in L^2(\Omega,\omega)$ with norm one such that $\|(\mathcal{K}-\lambda I)\phi_\delta\|\leq\delta$. The observable $\phi_\delta$ is called a $\delta$-pseudoeigenfunction~\cite[p.~31]{trefethen2005spectra} and satisfies
\begin{equation}
\label{approx_cohernecy}
\|\mathcal{K}^n\phi_\delta-\lambda^n \phi_\delta\| = \mathcal{O}(n\delta)\quad \forall n\in\mathbb{N}.
\end{equation}
Like their eigenfunction counterparts, pseudoeigenfunctions encode significant information about the underlying dynamical system \cite{mezicAMS}, such as the global stability of equilibria \cite{mauroy2016global} and ergodic partitions \cite{budivsic2012applied}. However, they do not diagonalise the Koopman operator.

\subsection{Spectral measures and diagonalization}

The spectral theorem \cite[Thm. X.4.11]{conway2019course} provides a diagonalization of $\mathcal{K}$ using a \textit{projection-valued measure} $\mathcal{E}$ supported on $\mathrm{Sp}(\mathcal{K})\subset\mathbb{T}$. To each Borel subset $S\subset\mathbb{T}$, $\mathcal{E}$ associates an orthogonal projection $\mathcal{E}(S)$, which one should think of as the spectral content of $\K$ on $S$. In particular, we have the following decompositions:
\begin{equation}
\label{KMD_proj_meas_diag}
g=\left(\int_\mathbb{T} 1\dd\mathcal{E}(\lambda)\right)g \qquad\text{and}\qquad \mathcal{K}g=\left(\int_\mathbb{T} \lambda\ \mathrm{d}\mathcal{E}(\lambda)\right)g \quad \forall g\in L^2(\Omega,\omega).
\end{equation}
Moreover, the analogue of~\cref{eq:perfectly_coherent} is provided by a functional calculus on ${\rm Sp}(\mathcal{K})$, so that
$$
g(\xv_n)= [\mathcal{K}^ng](\xv_0)= \left[\left(\int_\mathbb{T} \lambda^n\ \mathrm{d}\mathcal{E}(\lambda)\right)g\right](\xv_0)\quad\forall g\in L^2(\Omega,\omega).
$$
While $\mathcal{E}$ associates invariant subspaces of $\mathcal{K}$ to measurable subsets of its spectrum, $\mathcal{E}$ is unable to associate coherent modes in $\L^2(\Omega,\omega)$ to points in the continuous spectrum.

The \textit{scalar-valued} spectral measure of $\mathcal{K}$ with respect to $g\in L^2(\Omega,\omega)$ with $\|g\| = 1$ is a probability measure supported on $\mathbb{T}$ and defined as $\mu_g(S)=\langle \mathcal{E}(S)g,g \rangle$. Following a change of variables $\lambda=\exp(i\theta)$, we consider the measures $\xi_{g}$ on the periodic interval $\pp$ so that $\xi_{g}(S)=\mu_{g}(\exp(iS))$ for $S\subset\pp$. Similarly, we make use of complex Borel measures $\xi_{g,f}(S)=\langle \mathcal{E}(\exp(iS))g,f\rangle$, for $f,g\in L^2(\Omega,\omega)$. The Lebesgue decompositions of the scalar spectral measures $\xi_g$ lead to a partition of the Hilbert space $L^2(\Omega,\omega)$ into closed subspaces $\mathcal{H}_{\mathrm{pp}}$ and $\mathcal{H}_{\mathrm{c}}$. Dynamics in $\mathcal{H}_{\mathrm{pp}}$ can be resolved via a discrete KMD in \cref{eq:classic_KMD}, whereas those in $\mathcal{H}_{\mathrm{c}}$ cannot. These subspaces find interpretations across various applications, including fluid mechanics \cite{mezic2013analysis}, anomalous transport \cite{zaslavsky2002chaos}, and trajectory invariants/exponents \cite{kantz2004nonlinear}.

\subsection{Generalized eigenfunctions}
\label{sec:rigged_H}

To construct a KMD from coherent features, we replace the projection-valued measure with a generalized eigenfunction expansion. Generalized eigenfunctions are distributions on a subspace $\mathcal{S}\subset L^2(\Omega,\omega)$ of test functions, where we require that $\mathcal{K}\mathcal{S}\subset\mathcal{S}$ and $\mathcal{S}$ admits a dense, continuous embedding into $L^2(\Omega,\omega)$ as a complete topological vector space. The dual of $L^2(\Omega,\omega)$ then embeds densely and continuously into the space $\mathcal{S}^*$ of continuous linear functionals on $\mathcal{S}$. By identifying $L^2(\Omega,\omega)$ with its dual, one obtains a triple of dense, continuous embeddings $\mathcal{S}\subset L^2(\Omega,\omega)\subset\mathcal{S}^*$, referred to as a \textit{rigged Hilbert space}. Since the inner product $\langle\cdot,\cdot\rangle$ on $L^2(\Omega,\omega)$ is conjugate linear in its second argument, the duality pairing $\langle \psi | \phi\rangle$ for $\psi\in\mathcal{S}^*$ and $\phi\in\mathcal{S}$ is equal to $\langle \phi, \overline{\psi} \rangle$ when both $\psi$ and $\phi$ are in $L^2(\Omega,\omega)$. Given $\psi\in\mathcal{S}^*$, it is useful to define its conjugate $\psi^*\in\mathcal{S}^*$ by
$\langle \psi^* | \phi\rangle=\overline{\langle \psi| \overline\phi\rangle}$
for $\phi\in\mathcal{S}$. The Koopman operator has a continuous extension to $\mathcal{S}^*$ as
$\smash{\langle \mathcal{K}\psi|\phi \rangle=
\langle \psi|\overline{\mathcal{K}^*\overline{\phi}} \rangle=
\langle \psi|\mathcal{K}^*\phi \rangle}$,
where the second equality holds because $\mathcal{K}$ is a composition operator.

A functional $\psi \in \mathcal{S}^*$ is called a \textit{generalized eigenfunction} of $\mathcal{K}$ if there is a $\lambda\in\mathbb{C}$ such that\vspace{-4mm}
\begin{equation}
\smash{\langle \mathcal{K}\psi | \phi \rangle = \lambda \langle \psi | \phi \rangle \qquad \forall\phi \in \mathcal{S}}.\label{eqn:def_geneig}
\end{equation}
By the density of $\mathcal{S}$ in $\smash{L^2(\Omega,\omega)}$, $\smash{u \in L^2(\Omega,\omega)}$ is a generalized eigenfunction if and only if it is an eigenfunction. From \cref{eqn:def_geneig}, generalized eigenfunctions are coherent because
\begin{equation}
\langle \mathcal{K}\psi | \phi \rangle = \lambda^n \langle \psi | \phi \rangle \qquad \text{for all} \qquad \phi \in \mathcal{S}\text{ and }n\in\mathbb{Z}.
\end{equation}
To achieve a \textit{modal decomposition} using generalized eigenfunctions, we require that $\mathcal{S}$ also be a countably Hilbert nuclear space \cite[Ch.~1.3]{gel4generalized}. Then, there is a complete set of generalized eigenfunctions $\{g_{\theta,k}:\theta\in\pp,k\in\mathcal{J}\}\subset\mathcal{S}^*$ such that for any $g\in \mathcal{S}$, it holds in $L^2(\Omega,\omega)$:
\begin{equation}
\label{eq:rigged_expansion}
g = \sum_{k\in \mathcal{J}}\int_\pp \!\!\!\!\! \langle g_{\varphi,k}^*|g\rangle g_{\varphi,k}\dd\xi_{g_k}(\varphi), \quad \K g = \sum_{k\in \mathcal{J}}\int_\pp \!\!\!\!\!\exp(i\varphi)\langle g_{\varphi,k}^*|g\rangle g_{\varphi,k}\dd\xi_{g_k}(\varphi).
\end{equation}
The spectral measures $\{\xi_{g_k}\}$ reflect a choice of basis and normalization for the generalized eigenspaces. The multiplicity of $\lambda=\exp(i\theta)\in{\rm Sp}(\mathcal{K})$ is the number of nonzero functionals $\{g_{\theta,k}\}$ and the index set $\mathcal{J}$ is, in general, at most countable. This expansion replaces the projection-valued measure expansion in \cref{KMD_proj_meas_diag} with a KMD, where the Koopman mode associated with $g_{\varphi,k}$ is the dual product $\langle g_{\varphi,k}^*|g\rangle$. For example, if each observable in $\gv = [g_1,\ldots,g_l]^T$ lies in $\mathcal{S}$, we recover the vectorized form of \cref{eq:rigged_expansion} as an analogue of the classic KMD in~\cref{eq:classic_KMD}:
$$
\gv = \!\sum_{k\in \mathcal{J}}\int_\pp \!\!\!\!\! \langle g_{\varphi,k}^*|\gv\rangle g_{\varphi,k}\dd\xi_{g_k}(\varphi), \quad\!\text{so that}\!\quad
\mathcal{K}^n\gv =\! \sum_{k\in \mathcal{J}}\int_\pp \!\!\!\!\!\exp(in\varphi) \langle g_{\varphi,k}^*|\gv\rangle g_{\varphi,k}\dd\xi_{g_k}(\varphi).
$$
Here, $\langle g_{\varphi,k}^*|\gv\rangle$ is the Koopman mode associated with the point $\lambda=\exp(in\varphi)\in{\rm Sp}(\mathcal{K})$.

\section{The Rigged DMD algorithm}
\label{sec:rigged_dmd_alg}

To approximate generalized Koopman modes and eigenfunctions from the snaphot data in~\cref{eq:snapshots}, we consider ``wave-packet'' observables of the form
\begin{equation}\label{eqn:wave_packet}
g_\theta^\epsilon=\sum_{k\in \mathcal{J}}\int_\pp \!\!\!\!\! K_\epsilon(\theta-\varphi)\, \langle g_{\varphi,k}^*|g\rangle g_{\varphi,k}\dd\xi_{g_k}(\varphi)\in L^2(\Omega,\omega), \qquad\text{where}\qquad g\in\mathcal{S}.
\end{equation}
This wave-packet approximates the part of \cref{eq:rigged_expansion} with spectral parameter $\lambda=\exp(i\theta)$ via convolution with $K_\epsilon$. The function $K_\epsilon$ is a highly localized smoothing kernel approximating a Dirac delta measure as $\epsilon\downarrow 0$. Therefore, the kernel filters out the contribution of unwanted modes to $g_\theta^\epsilon$ and amplifies the contribution of the desired modes at $\lambda=\exp(i\theta)$ for small $\epsilon$. In the limit as $\epsilon\downarrow 0$, we aim for $g_\theta^\epsilon$ to converge in a suitable sense (e.g., in $\mathcal{S}^*$) to a generalized eigenfunction of $\mathcal{K}$ associated with the point $\lambda=\exp(i\theta)$. Note that while the generalized modes are distributions in the rigged Hilbert space and may not be associated with any standard functions in $L^2(\Omega,\omega)$ (for example, see~\cref{sec:nonlinear_pendulum}), the wave-packet approximations in~\cref{eqn:wave_packet} are relatively well-behaved functions in $L^2(\Omega,\omega)$.

To compute $\smash{g_\theta^\epsilon}$, we exploit a connection between rational convolution kernels and the resolvent, $\smash{(\mathcal{K}-zI)^{-1}}$, of the Koopman operator. For example, let
\begin{equation}\label{eq:poisson_kernel}
K^{\rm (P)}_\epsilon(\varphi)=\frac{1}{2\pi}\frac{1-(1+\epsilon)^{-2}}{1+(1+\epsilon)^{-2}-2(1+\epsilon)^{-1}\cos(\varphi)}
\end{equation}
be the Poisson kernel for the unit disc~\cite[p.16]{katznelson2004introduction}. If we set $r=1/(1+\epsilon)<1$, then (see~\cref{sec:stones_formula})
\begin{equation}\label{eq:poisson_packet}
g_\theta^\epsilon = \frac{1}{4\pi}\left[(\mathcal{K}+re^{i\theta}I)(\mathcal{K}-re^{i\theta}I)^{-1} - (\mathcal{K}+r^{-1}e^{i\theta}I)(\mathcal{K}-r^{-1}e^{i\theta}I)^{-1}\right]g.
\end{equation}
Therefore, we can explicitly compute $g_\theta^\epsilon$ by sampling the resolvent at a few strategic points in the complex plane. Although the Poisson kernel produces low-order accurate approximations to the generalized eigenfunctions in general, we construct high-order accurate rational convolution kernels in \cref{sec:high_order_kernels} and analyze their convergence properties in \cref{sec:convergence_theorems}.

The question remains: How can we approximate $\smash{(\mathcal{K}-zI)^{-1}}$ from the snapshot data in \cref{eq:snapshots}? A first attempt may project $\K$ onto a finite-dimensional subspace of observables. This is equivalent to using Extended Dynamic Mode Decomposition (EDMD) \cite{williams2015data} in the large data limit $M\rightarrow\infty$, which generally does not work.

\begin{example}[Non-convergence of EDMD]\label{sec:EDMD_not_converge}
Consider the Hilbert space $l^2(\mathbb{Z})$ with orthonormal basis $\{e_j\}_{j=-\infty}^\infty$. Let $\K$ be the bilateral shift defined by
$
\K e_{j}=e_{j-1}.
$
This operator is unitary and $\spec(\K)=\mathbb{T}$. Such shifts commonly appear when considering the restriction of Koopman operators to subspaces of observables, and we will revisit this example in \cref{sec:example_integrable_systems}. If we restrict the operator to $\mathrm{span}\{e_{-N},\ldots,e_{N}\}$, the large data limit $M\rightarrow\infty$ of the EDMD matrix approximation of $\K$ is the finite $(2N+1)\times (2N+1)$ Jordan matrix
$$
\kdmd=
\begin{pmatrix}
0& 1& & \\
  &\ddots&\ddots & \\
  & &0& 1\\
 & & &0\\
\end{pmatrix},\quad\text{with}\quad(\kdmd-z\Iv)^{-1}=
\begin{pmatrix}
\frac{-1}{z}& \frac{-1}{z^2} &\hdots &\frac{-1}{z^k}\\
  &\ddots&\ddots & \vdots\\
  & &\frac{-1}{z}& \frac{-1}{z^2}\\
 & & &\frac{-1}{z}\\
\end{pmatrix}
$$
The spectrum of $\kdmd$ is $\{0\}$ and severely unstable. For example, if $0<|z|<1$, then
$
(\kdmd-z\Iv)^{-1}e_0=\sum_{j=1}^N -z^{-j}e_{-j}
$
grows exponentially as $N\rightarrow\infty$, whereas
$
(\K-zI)^{-1}e_0=\sum_{j=1}^\infty z^{j-1} e_j.
$\hfill $\blacksquare$
\end{example}

To rectify this, we use mpEDMD, which computes the \textit{unitary part} of a Galerkin projection of $\mathcal{K}$ onto a finite-dimensional subspace spanned by a dictionary $\{\psi_1,\ldots,\psi_{N}\} \subset L^2(\Omega,\omega)$. This projection is computed using the snapshot data to approximate integration with respect to $\omega$. Utilizing the unitary part is crucial for achieving convergence (see \cref{cor_mpEDMD_res_conv}). In \cref{sec:EDMD_not_converge}, mpEDMD replaces the Jordan matrix with a circulant matrix
$$
\begin{pmatrix}
0& 1& & \\
  &\ddots&\ddots & \\
  & &0& 1\\
 1 & & &0\\
\end{pmatrix},
$$
effectively placing a $1$ in the bottom-left corner. In general, mpEDMD achieves a unitary approximation by taking the unitary part of a polar decomposition.

\cref{alg:RiggedDMD} summarizes Rigged DMD, which can be split into two key stages (\cref{fig:rigged_DMD}):
\begin{itemize}
	\item\textbf{Stage A:} Use mpEDMD to compute a discretization of $\K$ and $g$, as detailed in \cref{sec:res_comp}. Step 2 represents the observable $g$ in terms of the mpEDMD eigenvectors, which are orthogonal with respect to a data-driven approximation of $\omega$.
	\item\textbf{Stage B:} Build wave-packet approximations through sampling the resolvent, as detailed in \cref{sec:stones_formula}. Step 3 solves a small linear system to construct high-order rational smoothing kernels. Step 4 builds the wave-packet approximation. Step 5 computes spectral measures through an inner product.
\end{itemize}

\begin{algorithm}[t]
\textbf{Input:} Snapshots $\{\xv^{(m)},\yv^{(m)}\}_{m=1}^M$, quadrature weights $\{w_m\}_{m=1}^{M}$, dictionary of observables $\{\psi_j\}_{j=1}^{N}$, $\{a_j\}_{j=1}^m$ with ${\rm Im}(a_j)>0$, smoothing parameter $\epsilon>0$, $\Theta\subset\pp$, observable $g$.\\
\vspace{-4mm}
\begin{tcolorbox}[enhanced,boxsep=3pt,
                  left=1pt,
                  right=1pt,colback=white,title={\textbf{Stage A:} Build discretizations of $\K$ and $g$.},
									colbacktitle=white,coltitle=black,frame style={white, dashed},
									borderline={0.5mm}{0mm}{blue!70!white,dashed},
                  ]
\begin{algorithmic}[1]
\vspace{-2mm}
\State Apply mpEDMD (\cref{alg:mpEDMD}) to compute $\Kv$, $\Vv$ (eigenvectors), $\mathbf{\Lambda}$ (eigenvalues).
\State Compute the vector $\gv = (\Wv^{1/2}\mathbf{\Psiv}_X \Vv)^{\dagger}\Wv^{1/2}\left(g(\xv^{(1)}),\ldots,g(\xv^{(M)})\right)^\top$.\label{alg:last-step}
\end{algorithmic}
\end{tcolorbox}

\begin{tcolorbox}[enhanced,boxsep=3pt,
                  left=1pt,
                  right=1pt,colback=white,title={\textbf{Stage B:} Build wave-packet approximations through sampling the resolvent.},
									colbacktitle=white,coltitle=black,frame style={white, dashed},
									borderline={0.5mm}{0mm}{blue!70!white,dashed},
                  ]
\begin{algorithmic}[1]
\vspace{-2mm}
\setcounterref{ALG@line}{alg:last-step}
\State Solve the $m\times m$ system \cref{eq:vandermonde_condition} for the residues $\alpha_1,\dots,\alpha_m\in\mathbb{C}$.
\State For each $\theta\in \Theta$ and $j=1,\ldots,m$, compute
$$
\gv_\theta^{(j,+)}=(\mathbf{\Lambda}+e^{i\theta-i\epsilon a_j})(\mathbf{\Lambda}-e^{i\theta-i\epsilon a_j})^{-1}\gv,\quad \gv_\theta^{(j,-)}=(\mathbf{\Lambda}+e^{i\theta-i\epsilon \overline{a_j}})(\mathbf{\Lambda}-e^{i\theta-i\epsilon \overline{a_j}})^{-1}\gv.
$$
(NB: No matrix multiplications or inverses are needed in this step since $\mathbf{\Lambda}$ is diagonal.)
\State For each $\theta\in \Theta$, compute
\begin{align*}
\widetilde{\gv}_\theta&=\frac{-1}{4\pi}\sum_{j=1}^{m}\left(\alpha_j\gv_\theta^{(j,+)}-\overline{\alpha_j}\gv_\theta^{(j,-)}\right)\in\mathbb{C}^N&&\text{(mpEDMD eigenvector coordinates)},\\
\gv_\theta &= \Vv\widetilde{\gv}_\theta\in\mathbb{C}^N&&\text{(original dictionary coordinates)}.
\end{align*}
\State \textit{Optional:} If spectral measures are desired, for each $\theta\in \Theta$, compute
$$
\xi(\theta)=\frac{-1}{2\pi}\sum_{j=1}^{m}\mathrm{Re}\left(\alpha_j\gv^*{\gv}_\theta^{(j,+)}\right)\in\mathbb{R}.
$$
\end{algorithmic}
\end{tcolorbox}

\textbf{Output:} Vectors $\{\gv_\theta:\theta\in\Theta\}$ such that each $\Psiv\gv_\theta \in L^2(\Omega,\omega)$ is a wave-packet approximation to a generalized eigenfunction of $\K$ corresponding to spectral parameter $\lambda=\exp(i\theta)$. If desired, smoothed spectral measures $\{\xi(\theta):\theta\in\Theta\}$.
\caption{The Rigged DMD algorithm for computing generalized eigenfunctions of $\K$.}
\label{alg:RiggedDMD}
\end{algorithm}

A few practical points are worth discussing:
\begin{itemize}
	\item We work in mpEDMD eigenvector coordinates. Step 4 requires no matrix multiplication, and it is very cheap to compute wave-packet approximations for different $\theta$.
	\item Wave-packets can be evaluated at any point $\xv\in\Omega$ through the feature map $\Psiv(\xv)=[\psi_1(\xv) \cdots \psi_{{N}}(\xv) ]$. For example, we can evaluate wave-packets at the snapshot data.
	\item Spectral measures are also cheap to evaluate at several $\theta$. We recommend choosing an equispaced $\theta$ grid for visualization, such as in \cref{fig:lorenz1}.
	\item If the state space dimension $d$ is large, it may be difficult to visualize $g_\theta^\epsilon$. Instead, we can visualize the Koopman modes $\langle g_{\varphi,k}^*|\gv\rangle$ for a collection of observables $\gv = [g_1,\ldots,g_l]^T$. This post-processing step is summarized in \cref{alg:RiggedDMD2} and can also be achieved via normalising by the smoothed spectral measures computed in \cref{alg:RiggedDMD}.
	\item There are three main parameters in the algorithm: $M$, $N$, and $\epsilon$. These parameters are often chosen adaptively in relation to one another. For a given dictionary and $N$, the error resulting from a finite $M$ corresponds to a quadrature error and can be controlled in various contexts \cite[Section 4.1.3]{colbrook2023multiverse}. For a given $\epsilon$, we can determine a posteriori, through the computation of residuals and ResDMD \cite[Appendix C]{colbrook2021rigorous}, whether the dictionary is sufficiently rich to approximate the resolvent.
\end{itemize}
To explain how and why Rigged DMD works, we describe stages A and B and analyze key notions of convergence in greater detail in~\cref{sec:res_comp} and~\cref{sec:stones_formula} below. The reader may safely skip to the examples in \cref{sec:examples}.

\begin{algorithm}[t]
\textbf{Input:} Snapshot data $\{\xv^{(m)},\yv^{(m)}\}_{m=1}^M$, quadrature weights $\{w_m\}_{m=1}^{M}$, dictionary of observables $\{\psi_j\}_{j=1}^{N}$, $\{a_j\}_{j=1}^m$ with ${\rm Im}(a_j)>0$, smoothing parameter $\epsilon>0$, angles $\Theta\subset\pp$, observables $\gv = [g_1,\ldots,g_l]^T$.\\
\vspace{-4mm}
\begin{algorithmic}[1]
\State Apply steps 1 -- 5 of Rigged DMD (\cref{alg:RiggedDMD}) for each observable $g_p$ to compute $
\widetilde{\gv}_\theta^{(p)}$ (where the superscript denotes dependence on $p$) for $p=1,\ldots,l$ and $\theta\in\Theta$.
\State For each $\theta\in\Theta$, compute the mean
$\smash{\widetilde{\gv}_\theta=\frac{1}{l}\sum_{p=1}^l\widetilde{\gv}_\theta^{(p)}}$
and the vector $\cv_\theta\in\mathbb{C}^l$ defined component-wise by
$
[\cv_\theta]_p=\widetilde{\gv}_\theta^*\widetilde{\gv}_\theta^{(p)}.
$
\end{algorithmic}
\textbf{Output:} Vectors $\{\cv_\theta:\theta\in\Theta\}$.
\caption{Post-processing of Rigged DMD to compute generalized Koopman modes.}
\label{alg:RiggedDMD2}
\end{algorithm}

\subsection{Connections with previous work}
\label{sec:connections_prev_work}
Generalized eigenfunctions of Perron--Frobenius operators, which are pre-adjoints of Koopman operators, are well-studied. The decay rates of correlation functions, known as Pollicott--Ruelle resonances \cite{pollicott1985rate,ruelle1986resonances}, correspond to a distinct type of generalized eigenvalue of the Perron--Frobenius operator. Numerous studies have explored these expansions for specific examples of maps \cite{antoniou1993spectral,hasegawa1993spectral,saphir1992spectral,gaspard1992diffusion,hasegawa1992decaying,gaspard1992r,tasaki1993deterministic,antoniou1997generalized,antoniou1993generalized,dorfle1985spectrum,antoniou1997resonances,suchanecki1996rigged}. See also work on isolated spectra of quasi-compact transfer operators \cite{keller1999stability,baladi2000positive,dellnitz2000isolated}. Froyland has developed numerical techniques for isolated spectra and eigenfunctions \cite{froyland1997computer,froyland2007ulam,froyland2008unwrapping,froyland2014detecting,crimmins2020fourier}. A key distinction between these works and ours is that Pollicott--Ruelle resonances typically do not lie on the unit circle, corresponding instead to extensions of Koopman operators via meromorphic extensions of the resolvent. Operators that admit such extensions are termed Fredholm--Riesz operators \cite{bandtlow1997resonances}. In contrast, our work achieves a generalized eigenfunction expansion with eigenvalues on the unit circle and makes no assumption beyond $\mathcal{K}$ being unitary. Recently, Bandtlow, Just, and Slipantschuk \cite{slipantschuk2020dynamic,bandtlow2023edmd} and Wormell \cite{wormell2023orthogonal} have found a significant link in certain systems: EDMD eigenvalues in the large subspace limit can correspond to these resonances. Meanwhile, generalized eigenfunctions associated with the Koopman operator side have been less explored, with Mezić's work on integrable Hamiltonian systems with one degree of freedom \cite{mezic2020spectrum} being a notable exception. We revisit this example in \cref{sec:example_integrable_systems}, providing an explicit expansion for any number of degrees of freedom using the Radon and Fourier transforms for the action and angle parts, respectively. A general-purpose algorithm for computing generalized eigenfunctions of Koopman operators from snapshot data in \cref{eq:snapshots} has yet to be developed. Rigged DMD addresses this gap.

Spectral measures are formed from one-dimensional projections of generalized eigenfunctions. Methods for computing spectral measures of Koopman operators have been developed in the last few years.
We have already mentioned ResDMD \cite{colbrook2021rigorous} and mpEDMD \cite{colbrook2023mpedmd}.
ResDMD computes smoothed approximations of spectral measures associated with general measure-preserving dynamical systems using either the resolvent operator or filters applied to correlations. The paper \cite{colbrook2021rigorous} includes explicit high-order convergence theorems for the computation of spectral measures in various senses, including the density of the continuous spectrum, spectral projections of subsets of the unit circle, and the discrete spectrum. Similar smoothing techniques can also be used for self-adjoint operators \cite{colbrook2021computing}. The spectral measures of the mpEDMD approximations converge weakly to their infinite-dimensional counterparts. Other methods include compactification for Koopman generators of continuous-time systems \cite{das2021reproducing,valva2023consistent}, the partitioning of state space to obtain periodic approximations \cite{govindarajan2019approximation,govindarajan2021approximation}, post-processing of the moments of spectral measures of ergodic systems using the Christoffel--Darboux kernel \cite{korda2020data}, and harmonic averaging and Welch's method for post-transient flows \cite{arbabi2017study}.

The aforementioned papers compute spectral measures, but it is not always clear how their techniques can be extended to generalized eigenfunctions. Nevertheless, we build on some of these methods. We use mpEDMD as a discretization method to construct data-driven approximations of the resolvent $(\K-zI)^{-1}$. \Cref{sec:EDMD_not_converge} demonstrates the necessity for using mpEDMD. We then use the resolvent to compute smoothed approximations of generalized eigenfunctions. This smoothing is analogous to ResDMD, which computed smoothed approximations of spectral measures.

\section{Data-driven computation of the resolvent}
\label{sec:res_comp}

To compute the resolvent $(\K-zI)^{-1}g$, we will apply mpEDMD \cite{colbrook2023mpedmd}, which enforces measure-preserving approximations of $\K$. The mpEDMD algorithm is simple and robust, with no tuning parameters. Importantly, one can show its convergence of spectral properties and the resolvent (see~\cref{sec:res_conv}).

\subsection{Extended DMD}
\label{sec:EDMD}

The starting point of EDMD \cite{williams2015data} is a dictionary $\{\psi_1,\ldots,\psi_{N}\}\subset L^2(\Omega,\omega) $, i.e., a list of observables.  These observables generate a finite-dimensional subspace $V_N=\mathrm{span}\{\psi_1,\ldots,\psi_{N}\}$. Utilizing the snapshot data in \cref{eq:snapshots}, EDMD constructs a matrix $\kdmd\in\mathbb{C}^{N\times N}$. This matrix acts on coefficients within expansions in terms of the dictionary to approximate the action of the Koopman operator on $V_N$:
$$
{[\mathcal{K}\psi_j](\xv) = \psi_j(\Fv(\xv)) \approx \sum_{i=1}^{N} (\kdmd)_{ij} \psi_i(\xv)},\quad1\leq j\leq {N}.
$$
To build $\kdmd$, we view the snapshot data $\{\xv^{(m)}\}_{m=1}^{M}$ as quadrature nodes for integration with respect to $\omega$ and weights $\{w_m\}_{m=1}^{M}$.\footnote{There are typically three scenarios for convergence of the quadrature rule (see the discussion in \cite{colbrook2023multiverse}): high-order quadrature rules, suitable for small $d$ and when we are free to choose $\{\xv^{(m)}\}_{m=1}^{M}$; drawing $\{\xv^{(m)}\}_{m=1}^{M}$ from a single trajectory and setting $\omega_m=1/M$, suitable for ergodic systems; and drawing $\{\xv^{(m)}\}_{m=1}^{M}$ at random.} Define the vector-valued feature map $\Psiv(\xv)=\begin{bmatrix}\psi_1(\xv) & \cdots& \psi_{{N}}(\xv) \end{bmatrix}\in\mathbb{C}^{1\times {N}}$, the weight matrix $\Wv=\mathrm{diag}(w_1,\ldots,w_{M})$, and let
\begin{equation}
	\begin{split}
		\Psiv_X=\begin{pmatrix}
			       \Psiv(\xv^{(1)}) \\
			       \vdots              \\
			       \Psiv(\xv^{(M)})
		       \end{pmatrix}\in\mathbb{C}^{M\times N},\quad
		\Psiv_Y=\begin{pmatrix}
			       \Psiv(\yv^{(1)}) \\
			       \vdots              \\
			       \Psiv(\yv^{(M)})
		       \end{pmatrix}\in\mathbb{C}^{M\times N}.
		\label{psidef}
	\end{split}
\end{equation}
Assuming that the quadrature approximations are convergent, we have that
\begin{equation}
	\label{quad_convergence}
	\lim_{M\rightarrow\infty}[\Psiv_X^*\Wv\Psiv_X]_{jk} = \langle \psi_k,\psi_j \rangle\quad \text{ and }\quad \lim_{M\rightarrow\infty}[\Psiv_X^*\Wv\Psiv_Y]_{jk} = \langle \mathcal{K}\psi_k,\psi_j \rangle.
\end{equation}
Therefore, the EDMD matrix is defined as (without loss of generality, ${\rm rank}(\Wv^{1/2}\Psiv_X)=N$)
$$
	\kdmd= (\Wv^{1/2}\Psiv_X)^\dagger \Wv^{1/2}\Psiv_Y=(\Psiv_X^*\Wv\Psiv_X)^\dagger\Psiv_X^*\Wv\Psiv_Y,
$$
where `$\dagger$' denotes the pseudoinverse. Hence, $\kdmd$ approaches a matrix representation of $\mathcal{P}_{V_{N}}\mathcal{K}\mathcal{P}_{V_{N}}^*$ in the large data limit, where $\mathcal{P}_{V_{N}}$ denotes the orthogonal projection onto $V_{N}$.

\subsection{Measure-preserving EDMD}

The Gram matrix $\Gv=\Psiv_X^*\Wv\Psiv_X$ provides an approximation of the inner product $\langle \cdot,\cdot\rangle$ on $L^2(\Omega,\omega)$. Given $\hv,\gv\in\mathbb{C}^N$, we have that
\begin{equation}
\label{inner_product_G}
\hv^*\Gv\gv=\sum_{j,k=1}^N\overline{\hv_j}\gv_k\Gv_{j,k}\approx \sum_{j,k=1}^N  \overline{\hv_j}\gv_k\langle \psi_k,\psi_j \rangle=\langle \Psiv \gv,\Psiv \hv \rangle.
\end{equation}
If the convergence in \cref{quad_convergence} holds, the left-hand side of \cref{inner_product_G} converges to the right-hand side as $M\rightarrow\infty$. Hence, if $g=\Psiv\gv\in V_N$ and we approximate $\mathcal{K}$ on $V_N$ by a matrix $\Kv$,
$$
\|g\|^2\approx\gv^*\Gv\gv,\quad \|\mathcal{K}g\|^2\approx\|\Psiv \Kv\gv\|^2\approx \gv^*\Kv^*\Gv\Kv\gv.
$$
Since $\mathcal{K}$ is an isometry, $\|g\|^2=\|\mathcal{K}g\|^2.$ The mpEDMD algorithm combines EDMD with the constraint $\Kv^*\Gv\Kv=\Gv$. In a nutshell, we enforce that our Galerkin approximation is an isometry with respect to the learned, data-driven inner product induced by $\Gv$. 

\cref{alg:mpEDMD} summarizes the computation of $\Kv$ and its eigendecomposition. Since $\Kv$ is similar to a unitary matrix, its eigenvalues lie along the unit circle. Note that mpEDMD can be used with generic dictionary choices.

\begin{algorithm}[t]
\textbf{Input:} Snapshots $\{\xv^{(m)},\yv^{(m)}\}_{m=1}^M$, quadrature weights $\{w_m\}_{m=1}^{M}$, dictionary of observables $\{\psi_j\}_{j=1}^{N}$.\\
\vspace{-4mm}
\begin{algorithmic}[1]
\State Compute the matrices $\Psiv_X$ and $\Psiv_Y$ defined in \cref{psidef} and $\Wv=\mathrm{diag}(w_1,\ldots,w_{M})$.
\State Compute an economy pivoted QR decomposition $\Wv^{1/2}\Psiv_X=\Qv\Rv$, where $\Qv\in\mathbb{C}^{M\times N}$ and $\Rv\in\mathbb{C}^{N\times N}$.
\State Compute an SVD of $(\Rv^{-1})^{*}\Psiv_Y^*\Wv^{1/2}\Qv=\Uv_1\mathbf{\Sigma} \Uv_2^*$.
\State Compute the eigendecomposition $\Uv_2\Uv_1^*=\hat{\Vv}\mathbf{\Lambda} \hat{\Vv}^*$ (via a Schur decomposition).
\State Compute $\Kv=\Rv^{-1}\Uv_2\Uv_1^*\Rv$ and $\Vv=\Rv^{-1}\hat{\Vv}$.
\end{algorithmic} \textbf{Output:} Koopman matrix $\Kv$, with eigenvectors $\Vv$ and eigenvalues $\mathbf{\Lambda}$.
\caption{The mpEDMD algorithm \cite{colbrook2023mpedmd} using a QR decomposition (see \cite{colbrook2023multiverse}). For stability, we compute the eigendecomposition of $\Kv$ via the unitary matrix $\Uv_2\Uv_1^*$.}
\label{alg:mpEDMD}
\end{algorithm}

\subsection{Convergence to the resolvent}\label{sec:res_conv}

The following theorem shows that mpEDMD provides convergent approximations of $(\K-zI)^{-1}$ for $z\not\in\mathbb{T}$. Combined with the results of \cref{sec:convergence_theorems}, this theorem implies the convergence of Rigged DMD in the large data ($M\rightarrow\infty$) and large dictionary ($N\rightarrow \infty$) limits as $\epsilon\downarrow 0$. In practice, $\epsilon$ is chosen adaptively depending on $M$ and $N$ or vice versa.

\begin{theorem}
\label{cor_mpEDMD_res_conv}
Suppose that $\lim_{N\rightarrow\infty}\!\mathrm{dist}(h,V_{N})=0$ for all $h\in L^2(\Omega,\omega)$ and~\cref{quad_convergence} holds. Then for any $z\in\mathbb{C}\backslash\mathbb{T}$, $g\in L^2(\Omega,\omega)$ and $\gv_N\in\mathbb{C}^N$ with $\lim_{N\rightarrow\infty}\|g-\Psiv \gv_N\|=0$,
\begin{equation}
\label{mpEDMD_res_convergence}
\lim_{N\rightarrow\infty}\limsup_{M\rightarrow\infty} \|(\mathcal{K}-zI)^{-1}g-\Psiv(\Kv-z\Iv)^{-1}\gv_N\|=0.
\end{equation}
\end{theorem}

\begin{proof}
See \cref{sec:mpEDMD_app}.
\end{proof}

As we saw in \cref{sec:EDMD_not_converge}, the convergence in \cref{mpEDMD_res_convergence} need not hold for methods such as EDMD. In particular, convergence in the strong operator topology (the convergence of EDMD) need not imply convergence of the resolvent. The fact that mpEDMD provides a unitary approximation is crucial to \cref{cor_mpEDMD_res_conv}.

There are many natural settings where the dictionary does not span all of $L^2(\Omega,\omega)$ (for example, see~\cref{sec:cat_map}). Rather, in this setting the observable of interest belongs to a closed $\mathcal{K}$-invariant subspace of $L^2(\Omega,\omega)$ and the dictionary spans this space in an appropriate limit. \cref{cor_mpEDMD_res_conv} still holds if $L^2(\Omega,\omega)$ is replaced by any closed $\mathcal{K}$-invariant subspace thereof in both the hypotheses and the conclusions. That is, if $V_N$ spans this subspace in the limit $N\rightarrow\infty$, then the mpEDMD resolvent must also converge to the resolvent restricted to that subspace in the strong operator topology. Consequently, we can use the mpEDMD resolvent to approximate generalized Koopman modes for any observable in a closed $\mathcal{K}$-invariant subspace.

\section{Smooth wave-packet approximations to generalized eigenfunctions}
\label{sec:stones_formula}

To recover generalized eigenfunctions of $\K$ by sampling the resolvent $\smash{(\K-zI)^{-1}}$, we link $\smash{(\K-zI)^{-1}}$ to generalized eigenfunction in two steps. First, generalized eigenfunctions diagonalize the spectral measure on the rigged Hilbert space. For each Borel subset $S\subset\pp$~\cite{gadella2005measure},
\begin{equation}\label{eq:diag_pvm}
\mathcal{E}(S)g = \sum_{k\in \mathcal{J}}\int_S  \langle g_{\varphi,k}^*|g\rangle g_{\varphi,k} \dd\xi_{g_k}(\varphi) \qquad \text{for all} \qquad g\in\mathcal{S}.
\end{equation}
Second, the action of the projection-valued spectral measure $\mathcal{E}$ on observables can be recovered from the resolvent via the \textit{Carathéodory function} of $\mathcal{E}$, defined as
\begin{equation}
\label{eq:carath_def}
F_{\mathcal{E}}(z)=\int_{\pp}\frac{e^{i\varphi}+z}{e^{i\varphi}-z}\dd\mathcal{E}(\varphi)=(\K+zI)(\K-zI)^{-1},\quad |z|\neq 1.
\end{equation}
The second equality follows from the functional calculus of $\K$.\footnote{Carathéodory functions play a key role in the theory of orthogonal polynomials on the unit circle \cite{simon2005orthogonal}, where they are the CMV analogue of the so called $m$-function ($z\mapsto\langle (J-z)^{-1}e_1,e_1\rangle$ for a matrix $J$) in the theory of Jacobi matrices. For connections between the two, see \cite[Section 2.3]{simon2010szegHo}.} Writing $z=re^{i\theta}$ in modulus-argument form and assuming that $r=1/(1+\epsilon)<1$, one finds that (c.f.~\cref{eqn:wave_packet,eq:poisson_kernel,eq:poisson_packet}) 
\begin{equation}
\label{eq:carath_poisson}
\frac{1}{4\pi}\left[F_{\mathcal{E}}(re^{i\theta})-F_{\mathcal{E}}(r^{-1}e^{i\theta})\right]=\frac{1}{2\pi}\int_{\pp}\frac{(1-r^2)\dd\mathcal{E}(\varphi)}{1+r^2-2r\cos(\theta-\varphi)}=[K^{\rm (P)}_\epsilon*\mathcal{E}](\theta).
\end{equation}
The right-hand side is a convolution of $\mathcal{E}$ with the Poisson kernel defined in \cref{eq:poisson_kernel}. \Cref{fig:stone} illustrates the geometry of the left-hand side in the complex spectral domain.

A measure on $\pp$ is the weak limit of the real part of its Carathéodory function as $r\uparrow 1$ and Stone's classic formula~\cite{stone1932linear} recovers $\mathcal{E}$ by taking $r\uparrow 1$.
 Similarly, we can combine~\cref{eq:diag_pvm} and~\cref{eq:carath_poisson} to recover the generalized Koopman modes and eigenfunctions from $F_{\mathcal{E}}(z)$. We obtain convergence rates for wave-packet approximations to Koopman modes in the continuous spectrum under mild regularity conditions on spectral measures of $\mathcal{K}$. We denote the interior of a closed interval $\mathcal{I}$ by ${\rm int}(\mathcal{I})$ and write $g\in\mathcal{C}^{n,\alpha}(\mathcal{I})$ for a function with $n\in\mathbb{N}\cup\{0\}$ continuous derivatives on $\mathcal{I}$ with an $\alpha$-H\"older continuous $n$th derivative.

\begin{theorem}\label{thm:poisson_rates}
Suppose that the spectrum of $\K$ is absolutely continuous on the closed interval $\mathcal{I}\subset\pp$. Given $g,\phi\in\mathcal{S}$, let $\rho_{g,\phi}$ be the Radon--Nikodym derivative of the absolutely continuous component of $\xi_{g,\phi}$, and suppose that $\rho_{g,\phi}\in\mathcal{C}^{n,\alpha}(\mathcal{I})$ with $n\in\mathbb{N}\cup\{0\}$ and $\alpha\in[0,1]$. Denoting $g_\theta^\epsilon=\smash{[K^{\rm (P)}_\epsilon*\mathcal{E}](\theta)g}$, for any fixed $\theta_0\in{\rm int}(\mathcal{I})$, it holds that
$$
\left|\langle g_{\theta_0}^\epsilon|\overline{\phi}\rangle - \rho_g(\theta_0)\langle u_{\theta_0} | \overline{\phi}\rangle\right| = \mathcal{O}(\epsilon\,\llog) + \mathcal{O}(\epsilon^\alpha) \qquad\text{as}\qquad \epsilon\downarrow 0.
$$
Here, $u_\theta\in\mathcal{S}^*$ is a generalized eigenfunction of $\mathcal{K}$ satisfying~\cref{eqn:def_geneig} with $\lambda=\exp(i\theta)$. Moreover, in general, the convergence rate $\mathcal{O}(\epsilon\,\llog)$ is optimal for the Poisson kernel.
\end{theorem}

\begin{proof}
See~\cref{sec:conv_WPA_app}.
\end{proof}

The factor of $\rho_g(\theta)$ in~\cref{thm:poisson_rates} is equivalent to a renormalization of the generalized eigenfunctions in~\cref{eq:rigged_expansion} so that $\{g_\theta^\epsilon\}$ are orthonormal with respect to Lebesgue measure on the continuous spectrum of $\mathcal{K}$, rather than with respect to a spectral measure of $\mathcal{K}$, in the limit $\epsilon\downarrow 0$. If desired, $g_\theta^\epsilon$ can easily be rescaled with the smoothed measure computed in~\cref{alg:RiggedDMD}.

\begin{figure}
\centering
\raisebox{-0.5\height}{\includegraphics[width=0.35\textwidth]{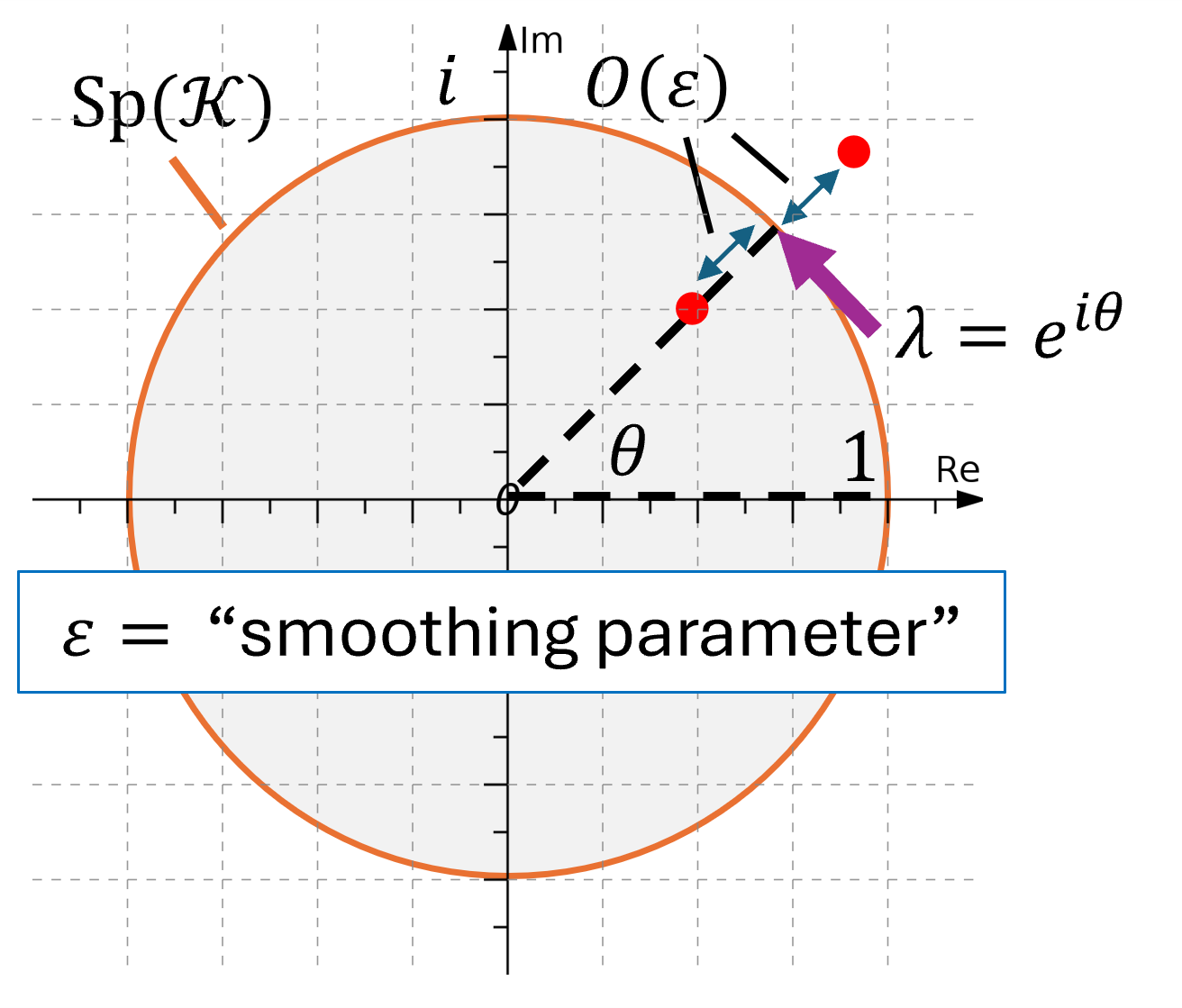}}\hspace{10mm}
\raisebox{-0.5\height}{\includegraphics[width=0.45\textwidth]{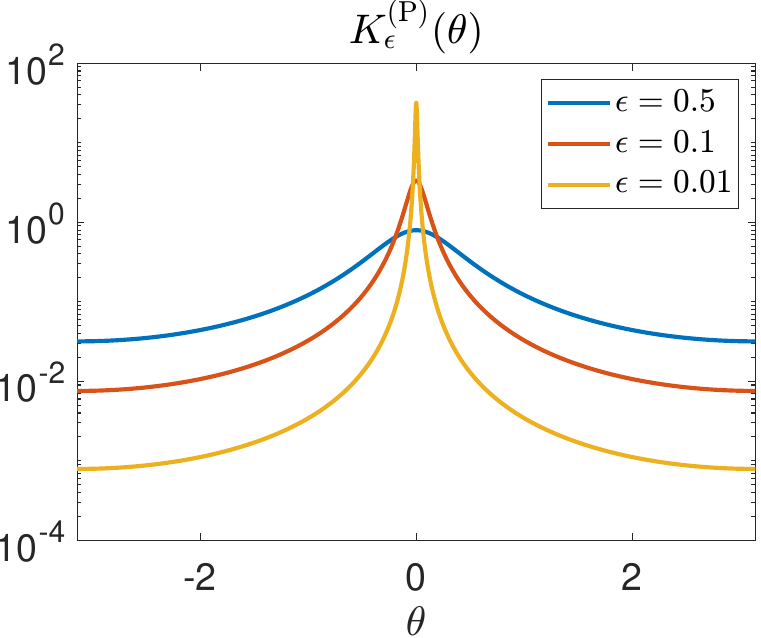}}
\caption{The essence of Stone's formula. Left: The points $z$ marked with {\color[rgb]{1,0,0}{$\bullet$}} are  where we evaluate $(\K-zI)^{-1}$, at a distance $\epsilon$ and $1-1/(1+\epsilon)$ from the unit circle. Right: The Poisson kernel for the unit disc. As $\epsilon\downarrow 0$, the Poisson kernel converges in the sense of distributions to a Dirac delta distribution.}
\label{fig:stone}
\end{figure}

In practice, the sublinear convergence rates associated with the Poisson kernel present a computational bottleneck. As noted in \cite[Figure 10]{colbrook2021rigorous}, as $\epsilon\downarrow 0$, the computational cost typically increases because a larger dictionary of observables is required to compute the resolvent $(\K-zI)^{-1}$ and, consequently, a larger collection of snapshot data is required. Therefore, keeping $\epsilon$ from becoming too small is crucial.

\subsection{High-order kernels}
\label{sec:high_order_kernels}

To alleviate the computational burden associated with slow convergence in the smoothing parameter $\epsilon$, we can replace the Poisson kernel with generalizations that achieve faster convergence rates as $\epsilon\downarrow 0$. Building on \cite{colbrook2021rigorous}, we now use Carathéodory functions to develop a general, streamlined framework for high-order convolution kernels on the unit circle. Our starting point is defining an $m$th order \textit{periodic kernel}.

\begin{definition}[$m$th order periodic kernel]
\label{def:unit_mth_kern_def}
Let $m\in\mathbb{N}$ and $\left\{K_{\epsilon}^\mathbb{T}:0<\epsilon\leq 1\right\}=\{K_{\epsilon}^\mathbb{T}\}$ be a family of continuous functions on the periodic interval $\pp$. We say that $\{K_{\epsilon}^\mathbb{T}\}$ is an $m$th order kernel for $\pp$ if it satisfies the following three properties:
\begin{itemize}
	\item[(i)] Normalized: $\int_{-\pi}^{\pi} K_{\epsilon}^\mathbb{T}(\theta)\dd\theta=1$.
	\item[(ii)] Approximately constant Fourier coefficients: There exists a constant $C_K$ such that for any integer $n=1,\ldots, m-1$,
\begin{equation}\label{eq:fourier_cond}
\left|\int_{-\pi}^{\pi}K_{\epsilon}^\mathbb{T}(-\theta)e^{in\theta}\dd\theta-1\right|\leq C_K \epsilon^m \llog.
\end{equation}
	\item[(iii)] Concentration around zero: For any $\theta\in\pp$ and $0<\epsilon\leq 1$,
	\begin{equation}\label{def:unitary_decay}
	\left|K_{\epsilon}^\mathbb{T}(\theta)\right|\leq \frac{C_K\epsilon^m}{(\epsilon+|\theta|)^{m+1}}.
	\end{equation}
\end{itemize}
\end{definition}

The name of this class of kernels is justified in \cref{sec:convergence_theorems}, where we provide high-order convergence theorems. We consider continuous kernels since this implies that for any complex Borel measure $\xi$ on $\pp$, the function
$$
\pp\ni\varphi\mapsto \int_{\pp}K_{\epsilon}^\mathbb{T}(\varphi-\theta)\dd\xi(\theta)=\int_{\pp}K_{\epsilon}^\mathbb{T}(\theta)\dd\xi(\varphi-\theta)
$$
is continuous. While the Poisson kernel for the unit disc in~\cref{eq:poisson_kernel} is a first-order kernel, our goal is to build high-order kernels that can be expressed in terms of the Carathéodory function. We construct these from high-order rational kernels designed for the \textit{real line} $\mathbb{R}$~\cite{colbrook2021computing}.

\begin{definition}[$m$th order kernel for $\mathbb{R}$]
\label{def:mth_order_kernel}
Let $m\in\mathbb{N}$. We say that a continuous bounded function $K$ on $\mathbb{R}$ is an $m$th order kernel (for $\mathbb{R}$) if it satisfies the following three properties:
\begin{itemize}
	\item[(i)] Normalized: $K$ is integrable and $\int_{\mathbb{R}}K(x)\dd x=1$.
	\item[(ii)]  Zero moments: $K(x)x^j$ is integrable and $\int_{\mathbb{R}}K(x)x^j\dd x=0$ for $j=1,\ldots,m-1$.
	\item[(iii)] Decay at $\pm\infty$: There is a constant $C_K$, independent of $x$, such that
	\begin{equation}\label{eq:decay_bound}
	\left|K(x)\right|\leq \frac{C_K}{(1+\left|x\right|)^{m+1}} \quad \forall x\in \mathbb{R}.
	\end{equation}
\end{itemize}
\end{definition}

The connection between \cref{def:unit_mth_kern_def} and \cref{def:mth_order_kernel} is the following proposition.

\begin{proposition}[Periodic summation of kernels]
\label{prop:per_summation}
Let $K$ be an $m$th order kernel for $\mathbb{R}$ and
$$
K_\epsilon^{\mathbb{T}}(\theta)=\sum_{n\in\mathbb{Z}} K_\epsilon(\theta+2\pi n)=\frac{1}{\epsilon}\sum_{n\in\mathbb{Z}} K\left(\frac{\theta+2\pi n}{\epsilon}\right), \quad 0<\epsilon\leq 1,
$$
be its periodic summation. Then $\{K_{\epsilon}^\mathbb{T}\}$ is an $m$th order kernel for $\pp$.
\end{proposition}

\begin{proof}
See \cref{sec:periodic_kernels_app}.
\end{proof}

Suppose now that $K$ is an $m$th order kernel for $\mathbb{R}$ of the form
\begin{equation}
\label{eq:rat_kern_R}
K(x)=\frac{1}{2\pi i}\sum_{j=1}^{m}\left(\frac{\alpha_j}{x-a_j}-\frac{\beta_j}{x-b_j}\right),
\end{equation}
where $a_1,\ldots,a_{m}$ are any distinct points in the upper half-plane and $b_1,\ldots,b_{m}$ are any distinct points in the lower half-plane. In \cite{colbrook2021computing}, it is shown that $K$ is an $m$th order kernel for $\mathbb{R}$ if
\begin{equation}\label{eq:vandermonde_condition}
\begin{pmatrix}
1 & \dots & 1 \\
a_1 & \dots & a_{m} \\
\vdots & \ddots & \vdots \\
a_1^{m-1} &  \dots & a_{m}^{m-1}
\end{pmatrix}
\begin{pmatrix}
\alpha_1 \\ \alpha_2\\ \vdots \\ \alpha_{m}
\end{pmatrix}\!
=
\begin{pmatrix}
1 & \dots & 1 \\
b_1 & \dots & b_{m} \\
\vdots & \ddots & \vdots \\
b_1^{m-1} &  \dots & b_{m}^{m-1}
\end{pmatrix}
\begin{pmatrix}
\beta_1 \\ \beta_2\\ \vdots \\ \beta_{m}
\end{pmatrix}
=\begin{pmatrix}
1 \\ 0 \\ \vdots \\0
\end{pmatrix}.
\end{equation}
The following proposition expresses the periodic summation of these kernels in terms of the Carathéodory function. This result provides the basis for stage B in \cref{alg:RiggedDMD}.

\begin{proposition}
\label{prop:carath_form}
Let $\{K_{\epsilon}^\mathbb{T}\}$ be the periodic summation (\cref{prop:per_summation}) of the $m$th order kernel for $\mathbb{R}$ in \cref{eq:rat_kern_R}, where the residues $\{\alpha_j,\beta_j\}$ satisfy \cref{eq:vandermonde_condition}. Then
\begin{equation}
\label{carathform1}
\left[K_\epsilon^{\mathbb{T}}\ast \mathcal{E}\right](\theta)=\frac{-1}{4\pi}\sum_{j=1}^{m}\left(\alpha_jF_\mathcal{E}(e^{i\theta-i\epsilon a_j})-\beta_jF_\mathcal{E}(e^{i\theta-i\epsilon b_j})\right).
\end{equation}
Moreover,  if $b_j = \overline{a_j}$ and $\beta_j=\overline{\alpha_j}$, then
\begin{align*}
\left[K_\epsilon^{\mathbb{T}}\ast\xi_g\right](\theta_0)&=\frac{-1}{2\pi}\sum_{j=1}^{m}{\rm Re}\left(\alpha_j \left\langle (\K-e^{i\theta-i\epsilon a_j}I)^{-1}g,(\K+e^{i\theta-i\epsilon a_j}I)^*g\right\rangle  \right).
\end{align*}
\end{proposition}

\begin{proof}
See \cref{sec:periodic_kernels_app}.
\end{proof}

For computations, we use equally spaced poles in $[-1,1]$ with
\begin{equation}\label{eq:equi_poles}
a_j=\frac{2j}{m+1}-1+i, \quad b_j=\overline{a_j}, \quad 1\leq j\leq m.
\end{equation}
We determine $\alpha_j=\overline{\beta_j}$ by solving \cref{eq:vandermonde_condition}. \cref{fig:kernels} shows these kernels for $m=1,\ldots,6$.

\begin{figure}
\centering
\includegraphics[width=0.45\textwidth]{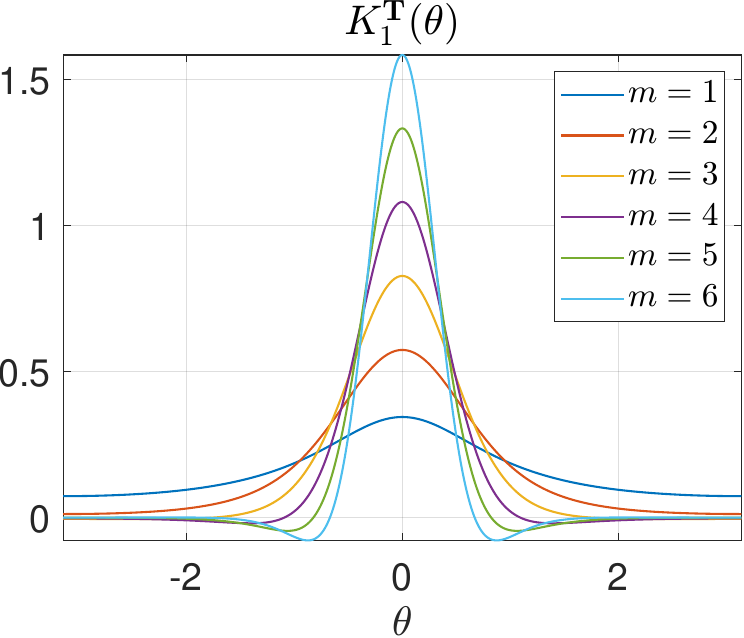}\hfill
\includegraphics[width=0.45\textwidth]{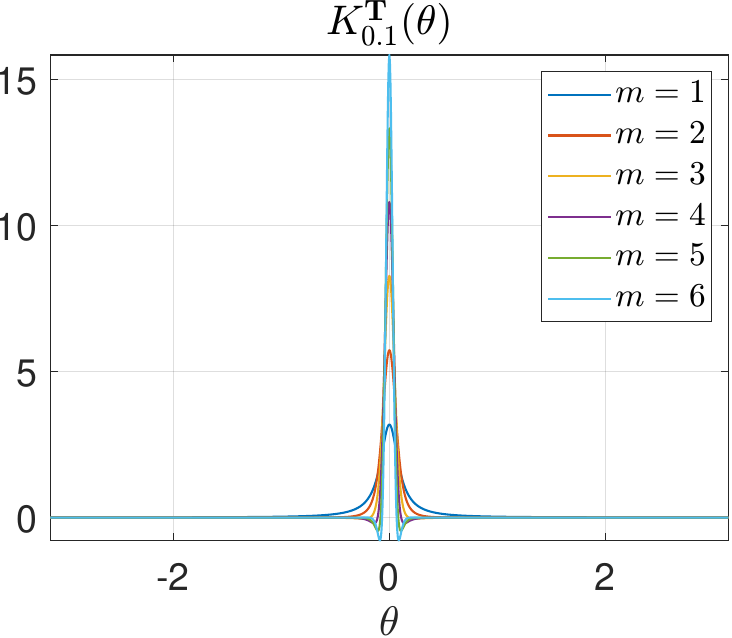}
 \caption{The $m$th order kernels on $\pp$ constructed via periodic summation of rational kernels with equispaced poles in $[-1,1]$. They localize around the origin as $\epsilon\downarrow 0$, approximating Dirac's delta measure.\label{fig:kernels}}
\end{figure}

\subsection{Convergence theorems}
\label{sec:convergence_theorems}

Returning to the generalized eigenfunction expansion in \cref{eq:rigged_expansion}, the $m$th-order kernel $K_\epsilon^{\mathbb{T}}$ induces a smoothed approximation of the KMD:
\begin{equation}
\label{eq:smoothedKMD}
\begin{split}
g_\theta^\epsilon=\left[K_\epsilon^{\mathbb{T}}\ast \mathcal{E}\right](\theta)g&=
\sum_{k\in \mathcal{J}}\int_\pp  \underbrace{K_\epsilon^{\mathbb{T}}(\theta-\varphi)}_{\text{sharp peak at $\theta$}}\,\,\underbrace{{\langle g_{\varphi,k}^*|g\rangle}}_{\text{K. mode}}\,\, \underbrace{g_{\varphi,k}}_{\text{gen. efun.}}\dd\xi_{g_k}(\varphi)\quad \forall g\in\mathcal{S}.
\end{split}
\end{equation}
The observable on the left-hand side is a wave-packet approximation of a generalized eigenfunction corresponding to the spectral parameter $\lambda=\exp(i\theta)$, as outlined in \cref{eqn:wave_packet}. Unlike Poisson wave packets, these wave packets are capable of exploiting smoothness in the spectral measures of $\mathcal{K}$ to provide high-order accurate approximations to the Koopman modes and generalized eigenfunctions associated with the continuous spectrum of $\K$. Beyond improved convergence, there are as least two other advantages of high-order kernels \cite[Section 5.1]{colbrook2021rigorous}: an improved stability to noise with suitably selected $\epsilon$; and improved localization of the kernel.

Throughout this subsection, we assume that $\mathcal{K}$ is unitary on a rigged Hilbert space $\mathcal{S}\subset L^2(\Omega,\omega)\subset\mathcal{S}^*$ with $\mathcal{K}\mathcal{S}\subset\mathcal{S}$. We use the notation $a\lesssim b$ to mean that $a\leq Cb$, where $C$ is a constant that can be made explicit and depends only on the constant $C_K$ in \cref{def:unit_mth_kern_def}, the order of the kernel $m$, and parameters $n+\alpha$ when considering the H\"older space $\mathcal{C}^{n,\alpha}(\mathcal{I})$.

\subsubsection{Convergence of generalized eigenfunctions}

The wave-packet approximations converge to generalized eigenfunctions of the Koopman operator in the sense of distributions on $\mathcal{S}$. This means that the \textit{approximate Koopman modes converge}. When the associated spectral measures are smooth, we can also characterize the rate of convergence. Given a complex Borel measure $\xi$ with finite variation, we let $\|\xi\|$ denote its total variation norm.

\begin{theorem}[Convergence of Koopman modes]\label{thm:conv_rates}
Let $\{K_{\epsilon}^\mathbb{T}\}$ be an $m$th order kernel for $\pp$. Suppose that for some $\theta_0\in\pp$ and $\eta\in(0,\pi)$, the spectrum of $\mathcal{K}$ is absolutely continuous on the closed interval $\mathcal{I}=[\theta_0-\eta,\theta_0+\eta]$. Given $g,\phi\in\mathcal{S}$, let $\rho_{g,\phi}$ be the Radon--Nikodym derivative of the absolutely continuous component of $\xi_{g,\phi}$, and suppose that $\rho_{g,\phi}\in\mathcal{C}^{n,\alpha}(\mathcal{I})$ with $n\in\mathbb{N}\cup\{0\}$ and $\alpha\in[0,1]$. Denoting $g_\theta^\epsilon=\smash{\left[K_\epsilon^{\mathbb{T}}\ast \mathcal{E}\right](\theta)g}$, it holds that
$$
\left|\langle g_{\theta_0}^\epsilon|\overline{\phi}\rangle {-} \rho_g(\theta_0)\langle u_{\theta_0} | \overline{\phi}\rangle\right| \!\lesssim\!
\frac{\epsilon^m\|\xi_{g,\phi}\|}{(\epsilon+\eta)^{m+1}}{+}\begin{cases}
\|\rho_{g,\phi}\|_{\mathcal{C}^{n,\alpha}(\mathcal{I})}(1+\eta^{-(n+\alpha)})\epsilon^{n+\alpha},\quad&\text{if }n+\alpha<m,\\
\|\rho_{g,\phi}\|_{\mathcal{C}^{m}(\mathcal{I})}(1+\eta^{-m})\epsilon^{m}\llog,\quad &\text{if }n+\alpha\geq m.
\end{cases}
$$
Here, $u_\theta\in\mathcal{S}^*$ is a generalized eigenfunction of $\mathcal{K}$ satisfying~\cref{eqn:def_geneig} with $\lambda=\exp(i\theta)$.
\end{theorem}

\begin{proof}
See~\cref{sec:conv_WPA_app}.
\end{proof}

The generalized eigenfunctions of $\mathcal{K}$ can often be identified with genuine functions on $\Omega$ that are smooth apart from certain singular features that prevent them from being in $L^2(\Omega,\omega)$. We now provide sufficient conditions for locally uniform convergence and convergence rates for wave packets computed with an $m$th order kernel. Given a measure $\xi$ on an interval $\mathcal{I}\subset\pp$ and a compact set $\mathcal{A}\subset\Omega$, we say that a function $u:\mathcal{A}\times \mathcal{I}\rightarrow \mathbb{C}$ is $\mathcal{A}$-uniformly $\xi$-integrable if $\int_\mathcal{I} |u(\xv,\theta)|\dd\xi(\theta)$ is uniformly bounded on $\mathcal{A}$. Similarly, we say that $u(\xv,\cdot)\in\mathcal{C}^{n,\alpha}(\mathcal{I})$ is $\mathcal{A}$-uniformly $(n,\alpha)$-H\"older if $\|u(\xv,\cdot)\|_{\mathcal{C}^{n,\alpha}(\mathcal{I})}$ is uniformly bounded on $\mathcal{A}$.

\begin{theorem}[Pointwise convergence of generalized eigenfunctions]
\label{thm:conv_rates2}
Suppose that the hypotheses of~\cref{thm:conv_rates} hold and let $\mathcal{A}\subset\Omega$ be compact. If $\rho_g u_\theta(\xv)$ is $\mathcal{A}$-uniformly $\xi_g$-integrable and $\rho_g(\theta) u_\theta(\xv)$ is continuous on $\mathcal{A}\times\mathcal{I}$, then
$$
\sup_{\xv\in\mathcal{A}}\left| g_{\theta_0}^\epsilon(\xv) - \rho_g(\theta_0) u_{\theta_0}(\xv)\right| \rightarrow 0 \qquad\text{as}\qquad \epsilon\downarrow 0.
$$
Moreover, if $\rho_g u_\theta(\xv)$ is $\mathcal{A}$-uniformly $(n,\alpha)$-H\"older with $n\in\mathbb{N}\cup\{0\}$ and $\alpha\in[0,1]$, then 
$$
\sup_{\xv\in\mathcal{A}}\left| g_{\theta_0}^\epsilon(\xv) - \rho_g(\theta_0) u_{\theta_0}(\xv)\right|=\mathcal{O}(\epsilon^m\llog) + \mathcal{O}(\epsilon^{n+\alpha}) \qquad \text{as}\qquad \epsilon\downarrow 0.
$$ 
\end{theorem}
\begin{proof}
See~\cref{sec:conv_WPA_app}.
\end{proof}

\paragraph{Convergence results without the rigged Hilbert space structure}

Even without the rigged Hilbert space structure, two convergence theorems are of interest. The functions $([K_{\epsilon}^\mathbb{T}\ast \mathcal{E}](\theta_0))g$ form an approximate eigenvector sequence. This means that in the limit $\epsilon\downarrow 0$, the approximate coherency in \cref{approx_cohernecy} is satisfied for arbitrarily small $\delta>0$.

\begin{theorem}[Approximate eigenfunction convergence]
\label{prop:evc_approx_atom}
Let $\{K_{\epsilon}^\mathbb{T}\}$ be an $m$th order kernel for $\pp$, $\theta_0\in\pp$ and $g\in L^2(\Omega,\omega)$. Let $g_\theta^\epsilon=\smash{\left[K_\epsilon^{\mathbb{T}}\ast \mathcal{E}\right](\theta)g}$ and suppose that
\begin{equation}
\label{eq:concentration_approx_evec_needed}
\liminf_{\epsilon\downarrow 0}\inf_{|\theta|<\eta\epsilon}{\epsilon|K_\epsilon^\mathbb{T}(\theta)|}>0\text{ for some $\eta>0$}\quad\text{and}\quad
\liminf_{c\downarrow 0}\frac{\xi_g([\theta_0-c,\theta_0+c])}{c^2}> 0.
\end{equation}
Then
$
\lim_{\epsilon\downarrow 0}{\|(\K\!-\!e^{i\theta_0}I)g_{\theta_0}^\epsilon\|}/{\|g_{\theta_0}^\epsilon\|}\!=\!0,
$
i.e., $g_{\theta_0}^\epsilon/\|g_{\theta_0}^\epsilon\|$ is an approximate eigenvector sequence.
\end{theorem}

\begin{proof}
See \cref{sec:conv_WPA_app}.
\end{proof}

The two conditions in \cref{eq:concentration_approx_evec_needed} are natural. The first condition ensures that the kernel is uniformly concentrated around $\theta=0$ and is satisfied by any kernel constructed from periodic summation in the manner of \cref{prop:per_summation}. The second condition guarantees that $\xi_g$ maintains a suitable non-vanishing concentration near $\theta_0$. In particular, this is satisfied if $\xi_g$ is absolutely continuous at $\theta_0$ with a continuous, non-zero Radon--Nikodym derivative there.

The following result establishes that if $\lambda = e^{i\theta}$ is an eigenvalue of $\mathcal{K}$, the wave-packet approximations converge to an appropriate eigenfunction after rescaling.

\begin{theorem}[Eigenfunction convergence]
\label{prop:evc_atom}
Let $\{K_{\epsilon}^\mathbb{T}\}$ be an $m$th order kernel for $\pp$ with
$
\liminf_{\epsilon\downarrow 0}{\epsilon|K_\epsilon^\mathbb{T}(0)|}>0.
$
Let $\theta_0\in\pp$ and $g\in L^2(\Omega,\omega)$. Then
$$
\lim_{\epsilon\downarrow 0}\frac{1}{K_\epsilon^\mathbb{T}(0)}[K_{\epsilon}^\mathbb{T}\ast \mathcal{E}](\theta_0)g=\begin{cases}
\mathcal{P}_{e^{i\theta_0}} g,\quad&\text{if }e^{i\theta_0}\text{ is an eigenvalue of }\K,\\
0,\quad&\text{otherwise},
\end{cases}
$$
where the convergence is in $L^2(\Omega,\omega)$.
\end{theorem}

\begin{proof}
See \cref{sec:conv_WPA_app}.
\end{proof}

\subsubsection{Convergence of spectral measures}

For completeness, we restate two theorems from \cite{colbrook2021rigorous} concerning the convergence of smoothed spectral measures. The first result, stated for a general, complex-valued Borel measure with finite variation, demonstrates that the spectral measures computed via Rigged DMD converge at the expected rates.

\begin{theorem}[Pointwise convergence]
\label{thm:unitary_pointwise_convergence}
Let $\{K_{\epsilon}^\mathbb{T}\}$ be an $m$th order kernel for $\pp$ and $\xi$ be a (possibly complex-valued) Borel measure on $\pp$ with total variation $\|\xi\|$. Suppose that for some $\theta_0\in\pp$ and $\eta\in(0,\pi)$, $\xi$ is absolutely continuous on the closed interval $\mathcal{I}=[\theta_0-\eta,\theta_0+\eta]$. Let $\rho$ be the Radon--Nikodym derivative of the absolutely continuous component of $\xi$, and suppose that $\rho\in\mathcal{C}^{n,\alpha}(\mathcal{I})$ with $n\in\mathbb{N}\cup\{0\}$ and $\alpha\in[0,1]$. Then,
$$
\left|\rho(\theta_0)-[K_\epsilon^\mathbb{T}\ast\xi](\theta_0)\right|\lesssim
\frac{\epsilon^m\|\xi\|}{(\epsilon+\eta)^{m+1}}+\begin{cases}
\|\rho\|_{\mathcal{C}^{n,\alpha}(\mathcal{I})}(1+\eta^{-(n+\alpha)})\epsilon^{n+\alpha},\quad&\text{if }n+\alpha<m,\\
\|\rho\|_{\mathcal{C}^{m}(\mathcal{I})}(1+\eta^{-m})\epsilon^{m}\llog,\quad &\text{if }n+\alpha\geq m.
\end{cases}
$$
\end{theorem}

To understand the next theorem, recall that if $\Phi:\mathbb{T}\rightarrow\mathbb{C}$ is a continuous function and $\K$ is unitary, then one can define a normal operator $\Phi(\K)$ via
$$
\Phi(\K)=\int_{\pp} \Phi(e^{i\varphi})\dd \mathcal{E}(\varphi)=\int_\mathbb{T}\Phi(\lambda)\dd \mathcal{E}(\lambda).
$$
The following theorem says that integration against $K_\epsilon^\mathbb{T}\ast\mathcal{E}$ approximates the functional calculus of $\K$. For example, the choice $\Phi(\lambda)=\lambda^n$ (with associated function $\phi(\theta)=\exp(in\theta)$ in angle coordinates) recovers $\K^n$ and implies convergence of the smoothed KMD in \cref{eq:smoothedKMD}.

\begin{theorem}[Convergence to the functional calculus]
\label{thm:unitary_weak_convergence}
Let $\{K_{\epsilon}^\mathbb{T}\}$ be an $m$th order kernel for $\pp$ and $\phi\in\mathcal{C}^{n,\alpha}(\pp)$ with $n\in\mathbb{N}\cup\{0\}$ and $\alpha\in[0,1]$. Then
$$
\left\|\int_{\pp}\phi(\varphi)\left(\dd\mathcal{E}(\varphi){-} [K_\epsilon^\mathbb{T}\ast\mathcal{E}](\varphi)\dd\varphi\right)\right\|{\lesssim}\begin{cases}
\|\phi\|_{\mathcal{C}^{n,\alpha}(\pp)}\epsilon^{n+\alpha},\quad&\text{if }n+\alpha<m,\\
\|\phi\|_{\mathcal{C}^{m}(\pp)}\epsilon^{m}\llog,\quad &\text{if }n+\alpha\geq m.
\end{cases}
$$
\end{theorem}

\Cref{thm:unitary_weak_convergence} implies similar bounds for the weak convergence of $K_\epsilon^\mathbb{T}\ast\xi_v$ to $\xi_v$, because
$$
\left|\int_{\pp}\!\!\phi(\varphi)\left(\!\dd\xi_v(\varphi)- [K_\epsilon^\mathbb{T}\ast\xi_v](\varphi)\dd\varphi\!\right)\right|\leq\left\|\int_{\pp}\!\!\phi(\varphi)\left(\!\dd\mathcal{E}(\varphi)- [K_\epsilon^\mathbb{T}\ast\mathcal{E}](\varphi)\dd\varphi\!\right)\right\|\!\|v\|^2.
$$

\section{A general and natural method to construct nuclear spaces}
\label{sec:nuclear_construction}

To apply the nuclear spectral theorem and Rigged DMD, we do not need to explicitly know what $\mathcal{S}$ is;
it suffices to know it exists with $g\in \mathcal{S}$. We now show that a natural nuclear space exists for Koopman operators. Suppose that $\{q_n:n\in\mathbb{N}\}$ is an orthonormal basis of $L^2(\Omega,\omega)$. This basis could be constructed using the Gram--Schmidt process, but we do not need it explicitly. For example, in \cref{sec:delay_embedd_example}, we implicitly use such a basis obtained via delay-embedding.

Associated with $\mathcal{K}$ is a unitary matrix $A$ acting on $l^2(\mathbb{N})$, where the canonical basis of $l^2(\mathbb{N})$ is identified with $\{q_n:n\in\mathbb{N}\}$. To construct $\mathcal{S}$, we consider weighted $l^2$ spaces
$$
l^2_\tau(\mathbb{N})=\left\{(x_j)_{j\in\mathbb{N}}:\sum_{j=1}^\infty |x_j|^2 \tau(j)<\infty\right\}\qquad
\text{where $\tau:\mathbb{N}\rightarrow\mathbb{R}_{>0}$ is a positive weight function.}
$$
The space $l^2_\tau(\mathbb{N})$ has an associated inner product and norm
$$
\langle u,v\rangle_\tau = \sum_{j=1}^\infty u_j\overline{v_j} \tau(j), \qquad \|u\|_\tau=\sqrt{\sum_{j=1}^\infty |u_j|^2 \tau(j)}.
$$
The following lemma will allow us to choose an appropriate sequence of weighted $l^2$ spaces.

\begin{lemma}
\label{lemma:weightedl2}
Suppose that the unitary matrix $A$ is sparse, meaning that for any $q_n$, $\mathcal{K}q_n$ can be written as a finite linear combination of basis vectors in $\{q_n:n\in\mathbb{N}\}$. Given a weight function $\tau:\mathbb{N}\rightarrow\mathbb{R}_{>0}$, there exists a weight function $\tau':\mathbb{N}\rightarrow\mathbb{R}_{>0}$ with $\tau'\geq \tau$ such that:
\begin{itemize}
	\item[(i)] The inclusion map
$
\iota:l^2_{\tau'}(\mathbb{N})\hookrightarrow l^2_\tau(\mathbb{N})
$
is nuclear;
\item[(ii)] If $u\in l^2_{\tau'}(\mathbb{N})$, then $A u \in l^2_\tau(\mathbb{N})$ with $\|A u\|_\tau\leq \|u\|_{\tau'}$.
\end{itemize}
\end{lemma}

\begin{proof}
See \cref{sec:unitary_RHS_app}.
\end{proof}

We repeatedly apply \cref{lemma:weightedl2} to construct a nuclear space. We begin with
$$
\mathcal{S}_0=L^2(\Omega,\omega)=\left\{u=\sum_{j=1}^\infty u_jq_j:(u_j)\in l^2(\mathbb{N})\right\}.
$$
Given
$$
\mathcal{S}_n=\left\{u=\sum_{j=1}^\infty u_jq_j:(u_j)\in l^2_{\tau_n}(\mathbb{N})\right\}\subset L^2(\Omega,\omega),
$$
we apply \cref{lemma:weightedl2} to obtain $\tau_{n+1}=\tau_n'$ and set
$$
\mathcal{S}_{n+1}=\left\{u=\sum_{j=1}^\infty u_jq_j:(u_j)\in l^2_{\tau_{n+1}}(\mathbb{N})\right\}\subset\mathcal{S}_n.
$$
By part (i) of \cref{lemma:weightedl2}, each inclusion map
$
\mathcal{S}_{n+1}\hookrightarrow \mathcal{S}_{n}
$
is nuclear.
Setting
$
\mathcal{S}=\cap_{n=1}^\infty\mathcal{S}_n,
$
we have
$
\mathrm{span}\{q_n:n\in\mathbb{N}\}\subset \mathcal{S}.
$
Moreover, $\mathcal{K}:\mathcal{S}\rightarrow \mathcal{S}$ is continuous due to part (ii) of \cref{lemma:weightedl2}.

\subsection{Example: Time-delay embedding}
\label{sec:construct_nuclear_space}

\cref{lemma:weightedl2} holds when we use time-delay embedding. Time-delay embedding is a popular method for DMD algorithms \cite{arbabi2017ergodic,brunton2017chaos,das2019delay,kamb2020time,pan2020structure} and corresponds to using a Krylov subspace. The technique is justified through Takens' embedding theorem \cite{takens2006detecting}, which says that under certain conditions, delay embedding a signal coordinate of the system can reconstruct the attractor of the original system up to a diffeomorphism.

For example, let $g_1,g_2,\ldots,g_k\in L^2(\Omega,\omega)$ and assume that $\mathcal{K}^ng_j$ are linearly independent for $n\in\mathbb{Z}_{\geq 0}$ and $j=1,\ldots,k$. We order the observables $g_1,\ldots,g_k,\mathcal{K}g,\ldots,\mathcal{K}g_k,\ldots$ as $\hat q_1,\hat q_2,\ldots$ and consider the subspace of $L^2(\Omega,\omega)$ generated by these observables. Let $q_1,q_2,\ldots$ be an orthonormal sequence such that
$$
\mathrm{span}\{q_1,\ldots,q_N\}= \mathrm{span}\{\hat q_1,\ldots,\hat q_N\}\quad \forall N\in\mathbb{N}.
$$
For example, if we consider time delays of a single variable, then the corresponding matrix $A$ is upper Hessenberg. For time-delays of $k$ variables, the matrix representation $A$ of $\mathcal{K}$ is sparse with $A_{ij}=\langle\mathcal{K}q_j,q_i\rangle=0$ if $i>j+k$. The proof of \cref{lemma:weightedl2} shows that given a weight function $\tau_n$, we must select $\tau_{n+1}$ such that
$$
\sum_{j=1}^\infty\sqrt{\frac{\tau_n(j)}{\tau_{n+1}(j)}}<\infty,\qquad\text{and}\qquad\sum_{j=1}^{l+k}\tau_n(j)\leq\tau_{n+1}(l).
$$
Beginning with $\tau_0\equiv 1$, we may select $\tau_n(j)=c_n j^{3n}$ for some constants $c_n$. It follows that
$$
\mathcal{S}=\left\{f=\sum_{j=1}^\infty f_jq_j:\lim_{j\rightarrow\infty}j^mf_j=0\,\,\forall m\in\mathbb{N}\right\}.
$$
In particular, $\mathcal{K}^ng_j\in\mathcal{S}$ for $n\in\mathbb{Z}_{\geq 0}$ and $j=1,\ldots,k$. It follows that $\mathcal{S}$ corresponds to observables that can be expanded in our Krylov subspace with sufficiently rapidly decaying expansion coefficients according to the powers of $\mathcal{K}$.

\section{Examples}
\label{sec:examples}

We consider four examples of Rigged DMD:
\begin{itemize}
	\item Systems with Lebesgue spectrum. These systems decompose into a direct sum of shift operators, the canonical unitary operators, an example where we can write down the generalized eigenfunctions analytically. The spectrum is absolutely continuous on $\mathbb{T}\backslash\{1\}$. We consider the cat map as an example, where the smoothed generalized eigenfunctions become increasingly oscillatory as $\epsilon \downarrow 0$, and we demonstrate the convergence of Rigged DMD.
	\item Integrable Hamiltonian systems. The associated Koopman operators decompose with a Kronecker structure. The spectrum is continuous on $\mathbb{T}\backslash\{1\}$. Again, we provide an explicit generalized eigenfunction expansion, where the generalized eigenfunctions are plane waves with singular support along hyperplanes of constant energy (with connections to the Radon and Fourier transforms for the action and angle parts, respectively). As an example, we treat the nonlinear pendulum.
	\item The Lorenz system. The spectrum is continuous on $\mathbb{T}\backslash\{1\}$ but the generalized eigenfunction expansion for this example is unknown. We utilize the method outlined in \cref{sec:construct_nuclear_space} to construct a rigged Hilbert space via time-delay embedding and also demonstrate the increased coherency provided by high-order kernels.
	\item A high-Reynolds-number flow where a discrete DMD expansion presents significant challenges. This system also illustrates the application of Rigged DMD to noisy datasets with large state-space dimensions.
\end{itemize}

\subsection{Systems with Lebesgue spectrum}

As our first example, we consider systems with Lebesgue spectrum. These systems generalize the bilateral shift in \cref{sec:EDMD_not_converge}. Just as differential operators such as the Schr\"odinger operator are canonical self-adjoint operators, the bilateral shift operator can be considered the canonical example of a unitary operator.

\begin{definition}
The dynamical system $(\Omega,\omega,\Fv)$ with $\omega(\Omega)=1$ has Lebesgue spectrum if there exists an orthonormal basis $\{1,f_{i,j}:i\in\mathcal{J},j\in\mathbb{Z}\}$ of $L^2(\Omega,\omega)$ such that
$$
\mathcal{K}f_{i,j}=f_{i,j-1},\quad i\in\mathcal{J},j\in\mathbb{Z}.
$$
Here, $\mathcal{J}$ is an index set whose cardinality is called the multiplicity of the Lebesgue spectrum.
\end{definition}

Many dynamical systems have Lebesgue spectrum, such as geodesic flows on manifolds of constant negative curvature \cite{gel1952geodesic} and K-automorphisms (e.g., Bernoulli schemes) \cite{cornfeld2012ergodic}. Having a Lebesgue spectrum implies that the system is mixing \cite[Theorem 10.4]{arnold1968ergodic}.

If the system $(\Omega,\omega,\Fv)$ has Lebesgue spectrum, $\mathcal{K}$ acts as a direct sum of bilateral shift operators on the space $\mathrm{span}\{1\}^\perp$. Hence, to compute the decomposition in~\cref{eq:rigged_expansion}, it is enough to study $\mathcal{K}$ restricted to each of the invariant subspaces $\mathcal{H}_i=\mathrm{Cl}(\mathrm{span}\{f_{i,j}:j\in\mathbb{Z}\})$. Hence, it is enough to study the bilateral shift operator on $l^2(\mathbb{Z})$.

\subsubsection{Generalized eigenfunctions of the bilateral shift operator}
\label{sec:gen_efun_shift}

We can give an explicit rigged Hilbert space for the bilateral shift and, hence, for systems with Lebesgue spectrum. Consider $\mathcal{H}=l^2(\mathbb{Z})$ and the bilateral shift operator $\mathcal{L}$ defined by its action on the canonical basis vectors $\mathcal{L}e_j=e_{j-1}$, for each $j\in\mathbb{Z}$. As our nuclear space, we take the space of rapidly decreasing sequences~\cite[p.~85]{gel4generalized}
$$
\mathcal{S}=\left\{a=\sum_{n\in\mathbb{Z}}a_ne_n:\lim_{|n|\rightarrow\infty}|n|^ma_n=0\,\,\forall m\in\mathbb{N}\right\}.
$$
Given $\lambda=e^{i\theta}\in\mathbb{T}$, the corresponding generalized eigenfunction is $u_\theta=\sum_{j\in\mathbb{Z}}e^{ij\theta}e_j$. Note that the generalized eigenfunctions are non-normalizable with expansion coefficients of constant amplitude.

The rigged Koopman mode decomposition in~\cref{eq:rigged_expansion} takes the explicit form
$$
a = \frac{1}{2\pi}\int_{\pp} \left\langle \sum_{n\in\mathbb{Z}}a_ne_n  ,u_\varphi\right\rangle\sum_{j\in\mathbb{Z}}e^{ij\varphi}e_j \dd \varphi=\frac{1}{2\pi}\int_{\pp} \sum_{n\in\mathbb{Z}}a_n e^{-in\varphi}\sum_{j\in\mathbb{Z}}e^{ij\varphi}e_j \dd \varphi.
$$
For example, for $a=e_n$, Rigged DMD computes the smoothed generalized eigenfunction
$$
\sum_{j\in\mathbb{Z}}e_j\frac{1}{2\pi}\int_\pp K_\epsilon^{\mathbb{T}}(\theta-\varphi) e^{i(j-n)\varphi} \dd\varphi.
$$
As $\epsilon\downarrow 0$, we recover $e^{-in\theta}u_\theta$, with the $j$th coefficient in the sum converging as follows:
$$
\frac{1}{2\pi}\int_\pp K_\epsilon^{\mathbb{T}}(\theta-\varphi) e^{i(j-n)\varphi} \dd\varphi\rightarrow e^{i(j-n)\theta}.
$$
Notice that, in the limit $\epsilon\downarrow 0$, the expansion of the smoothed generalized eigenfunction in the canonical basis decays more slowly, approximating the non-normalizable discrete Fourier modes that characterize the generalized eigenfunctions of the bilateral shift operator. 

\subsubsection{Example: Arnold's cat map}\label{sec:cat_map}

As an example of a system with Lebesgue spectrum, let $\Omega=[-\pi,\pi]_{\mathrm{per}}^2$, $\omega$ be the uniform measure and consider Arnold's cat map \cite{arnold1968ergodic}
$$
\Fv(\xv)=(2x+y,x+y)=\begin{pmatrix}
2&1\\
1&1
\end{pmatrix}\xv.
$$
The functions $\{\phi_{m,n}(\xv)=\exp(imx+iny):m,n\in\mathbb{Z}\}$ form an orthonormal basis of $L^2(\Omega,\omega)$. The constant function is left invariant under $\mathcal{K}$, whereas in general,
$$
\mathcal{K}\phi_{m,n}=\phi_{2m+n,m+n}.
$$
Since the matrix $(2,1;1,1)$ is invertible over $\mathbb{Z}$, we can partition $\{\phi_{m,n}:(m,n)\neq(0,0)\}$ into distinct orbits, which one can show are infinite. Consider $\phi_{p,0}$ with $p$ prime, then any basis function $\phi_{m,n}$ in the orbit $\{\mathcal{K}^j\phi_{p,0}:j\in\mathbb{Z}\}$ must have indices $m$ and $n$ divisible by $p$. It follows that there must be a countable infinity of orbits. Therefore, the cat map has a countably infinite Lebesgue spectrum. Moreover, as the basis $\smash{\{\phi_{m,n}(\xv)\}}$ consists of Fourier modes on $\Omega$, the analysis of the bilateral shift operator in \cref{sec:gen_efun_shift} shows that the smoothed generalized eigenfunctions computed by Rigged DMD become more oscillatory as $\epsilon\downarrow 0$. 

\begin{figure}
\centering
\includegraphics[width=0.32\textwidth]{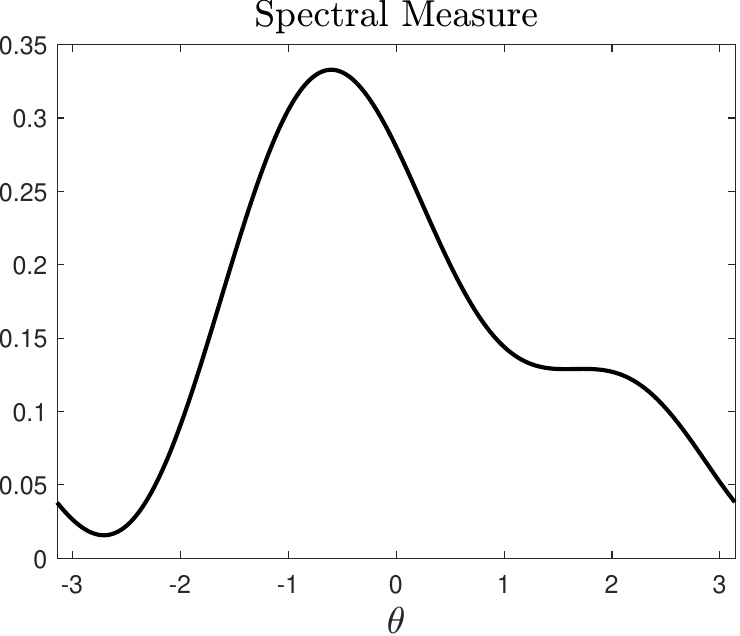}\hfill
\includegraphics[width=0.32\textwidth]{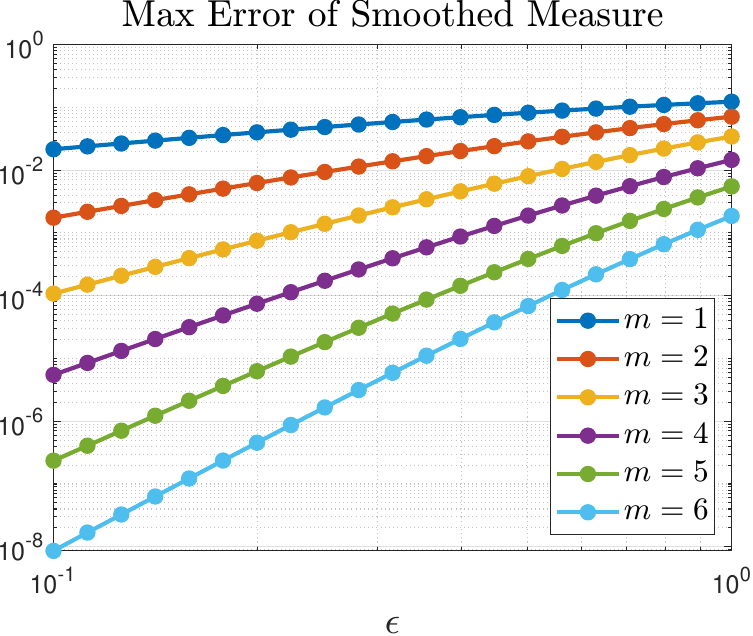}\hfill
\includegraphics[width=0.32\textwidth]{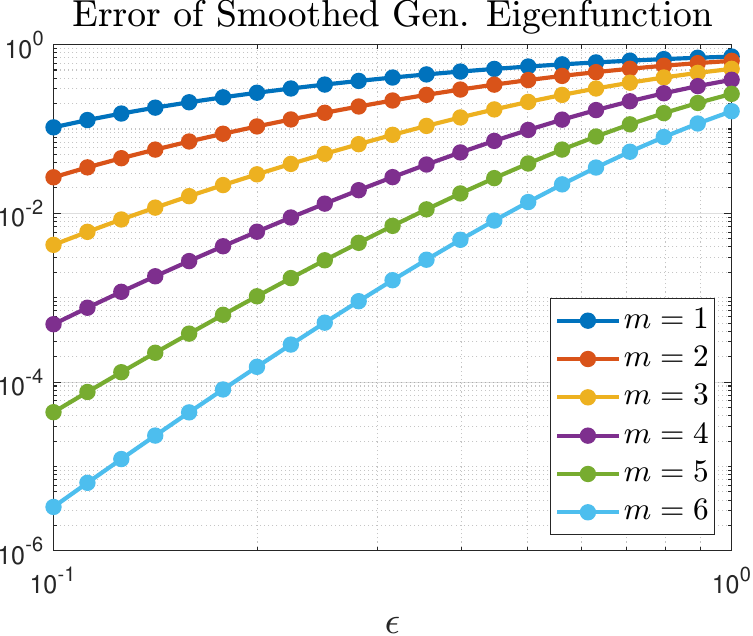}
\caption{Left: The density of the spectral measure $\xi_g$ for the cat map example. Middle: The $L^\infty(\pp)$ error of the smoothed approximation computed using the kernels in \cref{fig:kernels}. Right: The error of the smoothed generalized eigenfunctions of the cat map computed using \cref{alg:RiggedDMD}.\label{fig:catmap1}}
\end{figure}

\begin{figure}
\centering
\includegraphics[width=0.93\textwidth]{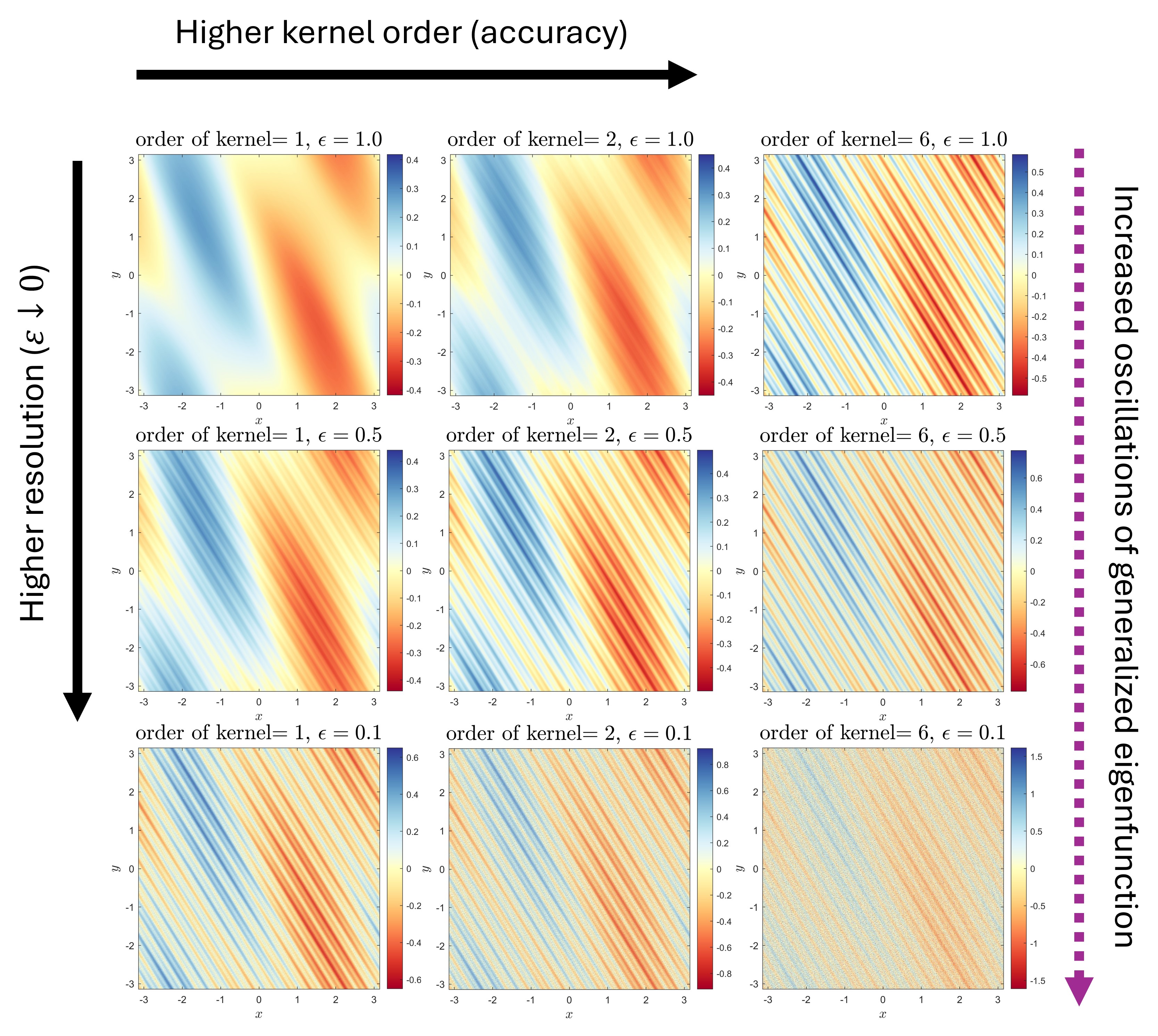}
\caption{The smoothed generalized eigenfunctions of the cat map computed using \cref{alg:RiggedDMD}.\label{fig:catmap3}}
\end{figure}

We collect data $\{\xv^{(m)},\yv^{(m)}\}$ along an equispaced $50\times 50$ grid for $\xv^{(m)}$ over $\Omega$, corresponding to $M=2500$. As a dictionary, we employ time-delay embedding with the function
$$
g=\sin(x) + \frac{1}{2}\sin(2x+y) + \frac{i}{4}\sin(5x+3y).
$$
This function is chosen because we can analytically compute the spectral measure and generalized eigenfunctions. \cref{fig:catmap1} shows the spectral measure and the maximum errors of the smoothed approximations computed using \cref{alg:RiggedDMD}. We see the convergence rates $\mathcal{O}(\epsilon^{m}\llog)$ predicted by the theorems in \cref{sec:convergence_theorems}. To demonstrate the convergence of the smoothed generalized eigenfunctions, we represent the eigenfunction at $\theta=1$ in the basis $\{g,\mathcal{K}g,\K^2g,\ldots\}$ and consider the relative $l^2$ norm error of the first five coefficients. These are also shown in \cref{fig:catmap1} and again demonstrate the high-order convergence. \cref{fig:catmap3} shows the corresponding smoothed generalized eigenfunctions. As $\epsilon\downarrow 0$, these functions become more oscillatory (as predicted by the above analysis). For the higher-order kernels, we see an increased resolution of the approximation for a given $\epsilon$.

\subsection{Integrable Hamiltonian systems}
\label{sec:example_integrable_systems}

The nonlinear pendulum has continuous spectra and is a well-known challenge for DMD methods \cite{lusch2018deep}. We consider general integrable Hamiltonian systems, before specializing to this example.
The Arnold--Liouville theorem states that if a Hamiltonian system with $n$ degrees of freedom has $n$ independent, Poisson commuting first integrals of motion, and if the constant energy manifolds are compact, then there exists a transformation to local action-angle coordinates in which (a) the transformed Hamiltonian is dependent only upon the action coordinates and (b) the angle coordinates evolve linearly in time \cite{arnol2013mathematical}. Suppose that our system can be transformed to global action-angle coordinates $\xv=(\Iv,\boldsymbol{\theta})\in\mathbb{R}^n\times[-\pi,\pi]_\mathrm{per}^n=\Omega$ such that $\dot{\Iv}=0, \dot{\boldsymbol{\theta}}=\Iv$. We sample the continuous dynamical system at discrete time steps of $\Delta t>0$. The associated discrete dynamical system has the map $\Fv(\Iv,\boldsymbol{\theta})=(\Iv,\boldsymbol{\theta}+\Delta t\Iv)$ and the Koopman operator is
$$
[\mathcal{K}g](\Iv,\boldsymbol{\theta}) = g(\Iv,\boldsymbol{\theta}+\Delta t\Iv), \qquad\text{where}\qquad g\in L^2(\Omega,\omega),
$$
and $\omega$ is the Lebesgue measure normalized by $(2\pi)^{-n}$.
We equip $L^2(\Omega,\omega)$ with the nuclear space $\mathcal{S}=S(\mathbb{R}^n)\times C^\infty([-\pi,\pi]_\mathrm{per}^n)$, where $S(\mathbb{R}^n)$ denotes the Schwartz space on $\mathbb{R}^n$.

Given $g\in\mathcal{S}$, we denote the $n$-dimensional Fourier series in the angle coordinates of $g$ by
$$
g(\Iv,\boldsymbol{\theta})=\sum_{\kv\in\mathbb{Z}^n} \hat g_\kv(\Iv) e^{i\kv\cdot\boldsymbol{\theta}}, \qquad\text{where}\qquad \hat g_\kv(\Iv) = \frac{1}{(2\pi)^n}\int_{[-\pi,\pi]_\mathrm{per}^n} g(\Iv,\boldsymbol{\theta})e^{-i\kv\cdot\boldsymbol{\theta}}\dd{\boldsymbol{\theta}}.
$$
In the Fourier basis for the angle coordinates, $\mathcal{K}$ acts as a multiplication operator, because
\begin{equation}
\label{FT_diag_ham}
\widehat{[\mathcal{K}g]}_\kv(\Iv) = \exp(-i\Delta t \kv\cdot\Iv) \hat g_\kv(\Iv), \qquad \text{for each} \qquad \kv \in\mathbb{Z}^n.
\end{equation}
Examining the resolvent of this multiplication operator, it is straightforward to see that the spectrum of $\mathcal{K}$ is continuous for all $\lambda\in\mathbb{T}$ except for an eigenvalue of (countably) infinite multiplicity at $\lambda=1$, corresponding to observables with no angular dependence. We, therefore, study each Fourier mode separately.

When $\kv=0$, the Koopman operators acts as the identity and, consequently, every angle-independent observable is an eigenfunction. Let $\{\phi_j\}_{j=1}^\infty\subset S(\mathbb{R}^n)$ be an orthonormal basis of $L^2(\mathbb{R}^n)$. The zeroth mode can be recovered by the generalized eigenfunction expansion
$$
g_0(\Iv)=\sum_{j=1}^\infty \underbrace{\int_{\mathbb{R}^n}g_0(\Iv')\overline{\phi_j(\Iv')}\dd \Iv'}_{=\langle \phi_j^*|g\rangle}\phi_j(\Iv)=\sum_{j=1}^\infty\int_\pp \langle \phi_j^*|g\rangle\phi_j(\Iv)\delta(\varphi)\dd\varphi.
$$
When $\kv\neq 0$ and $\tau\in\mathbb{R}$, we define the $(n-1)$-dimensional hyperplane
$$
H_{\kv}(\tau)=\{\Iv\in\mathbb{R}^n:\Delta t \Iv\cdot\kv=\tau\},
$$
and $Q_\kv(\Iv)$ the projection of $\Iv$ onto its component perpendicular to $\kv$. The Koopman operator acts as multiplication by a constant along each hyperplane $H_{\kv}(\tau)$ but is a nontrivial multiplication operator along their orthogonal complements.
Let $\{\psi_j^{(\kv)}\}_{j=1}^\infty\subset S(H_{\kv}(0))$ be an orthonormal basis of the space $L^2(H_{\kv}(0))$ equipped with the standard surface measure $\sigma$. For $\lambda=e^{i\theta}\in\mathbb{T}$, $m\in\mathbb{Z}$ and $j\in\mathbb{N}$, we define the distribution $u_\theta^{(\kv,m,j)}$ via
\begin{equation}
\label{meaning_of_u}
\left\langle {u_\theta^{(\kv,m,j)}}^*|g\right\rangle=\int_{H_{\kv}(\theta+2\pi m)} \hat g_\kv(\Iv)\overline{\psi_j^{(\kv)}(Q_\kv(\Iv))}\dd \sigma(\Iv).
\end{equation}
This integral is the Radon transform of $\hat g_\kv(\Iv)$ \cite[Chapter 6]{epstein2007introduction} with an additional integration against the function $\overline{\psi_j^{(\kv)}(Q_\kv(\Iv))}$.
Using the relation \eqref{FT_diag_ham}, we see that $u_\theta^{(\kv,m,j)}$ is a generalized eigenfunction corresponding to $\lambda=e^{i\theta}$. Formally, we have
\begin{equation}
\label{Ham_sing_supp}
u_\theta^{(\kv,m,j)}(\Iv,\boldsymbol{\theta})=\frac{1}{(2\pi)^n}\delta(\theta+2\pi m-\Delta t\kv\cdot\Iv)\psi_j^{(\kv)}(Q_\kv(\Iv))e^{i\kv\cdot\boldsymbol{\theta}},
\end{equation}
where the Dirac-delta distribution should be understood in terms of the integration along hyperplanes in \cref{meaning_of_u}. It is straightforward to show that
$$
\hat g_\kv(\Iv)e^{i\kv\cdot\boldsymbol{\theta}}=\sum_{m=-\infty}^\infty\sum_{j=1}^\infty \int_\pp \langle u_\varphi^{(\kv,m,j)*}|g\rangle u_\varphi^{(\kv,m,j)}(\Iv,\boldsymbol{\theta}) \dd \varphi.
$$
Hence, the Koopman mode decomposition of $g\in\mathcal{S}$ is
$$
g(\Iv,\boldsymbol{\theta}) \!=\! \sum_{j=1}^\infty\!\int_\pp\!\!\! \langle \phi_j^*|g\rangle\phi_j(\Iv)\delta(\varphi)\dd\varphi+\!\!\!\sum_{\kv\in\mathbb{Z}^n\backslash\{0\}}\sum_{m=-\infty}^\infty\sum_{j=1}^\infty \int_\pp\!\!\! \langle u_\varphi^{(\kv,m,j)*}|g\rangle u_\varphi^{(\kv,m,j)}(\Iv,\boldsymbol{\theta}) \dd \varphi.
$$
In particular, the generalized eigenfunctions consist of eigenfunctions in the usual sense (in $L^2(\Omega,\omega)$) corresponding to $\lambda=1$ that depend only on $\Iv$, and generalized eigenfunctions of the form in \cref{Ham_sing_supp} corresponding to any $\lambda\in\mathbb{T}$. These latter generalized eigenfunctions are plane waves with singular support along hyperplanes of constant directions of action.

\subsubsection{Nonlinear pendulum}\label{sec:nonlinear_pendulum}

We consider the dynamical system of the nonlinear pendulum. The state variables $\xv=(x_1,x_2)$ are governed by the following coupled ODEs: 
$$
\dot{x_1}=x_2,\quad \dot{x_2}=-\sin(x_1),\quad\text{ with}\quad \Omega=[-\pi,\pi]_{\mathrm{per}}\times \mathbb{R},
$$
where $\omega$ is the standard Lebesgue measure on $\Omega$. We consider the corresponding discrete-time dynamical system by sampling with a time-step $\Delta_t=1$. This system is Hamiltonian with one degree of freedom. A special case of the spectral decomposition we derived in \cref{sec:example_integrable_systems} was derived in \cite{mezic2020spectrum} for this system.

We employ time-delay embedding $\{g,\mathcal{K}g,\ldots, \mathcal{K}^{299}g\}$ with $g(x_1,x_2)=\exp(ix_1)/\cosh(x_2)$, and collect data on a 500-point equispaced grid in both $x_1$ and $x_2$ directions. Trajectory data is generated using MATLAB's \texttt{ode45}, based on initial conditions from this grid. We apply a trapezoidal quadrature rule with super-algebraic convergence in mpEDMD~\cite{trefethen2014exponentially,trefethen2022exactness}. \Cref{fig:pendulum1} shows the spectral measure and generalized eigenfunctions computed using Rigged DMD with smoothing parameter $\epsilon=0.05$ and the $6$th order kernel from \cref{fig:kernels}. The singular support and plane wave structure of the generalized eigenfunctions for $\lambda\neq 0$ are clearly visible.

\begin{figure}
\centering
\includegraphics[width=0.9\textwidth]{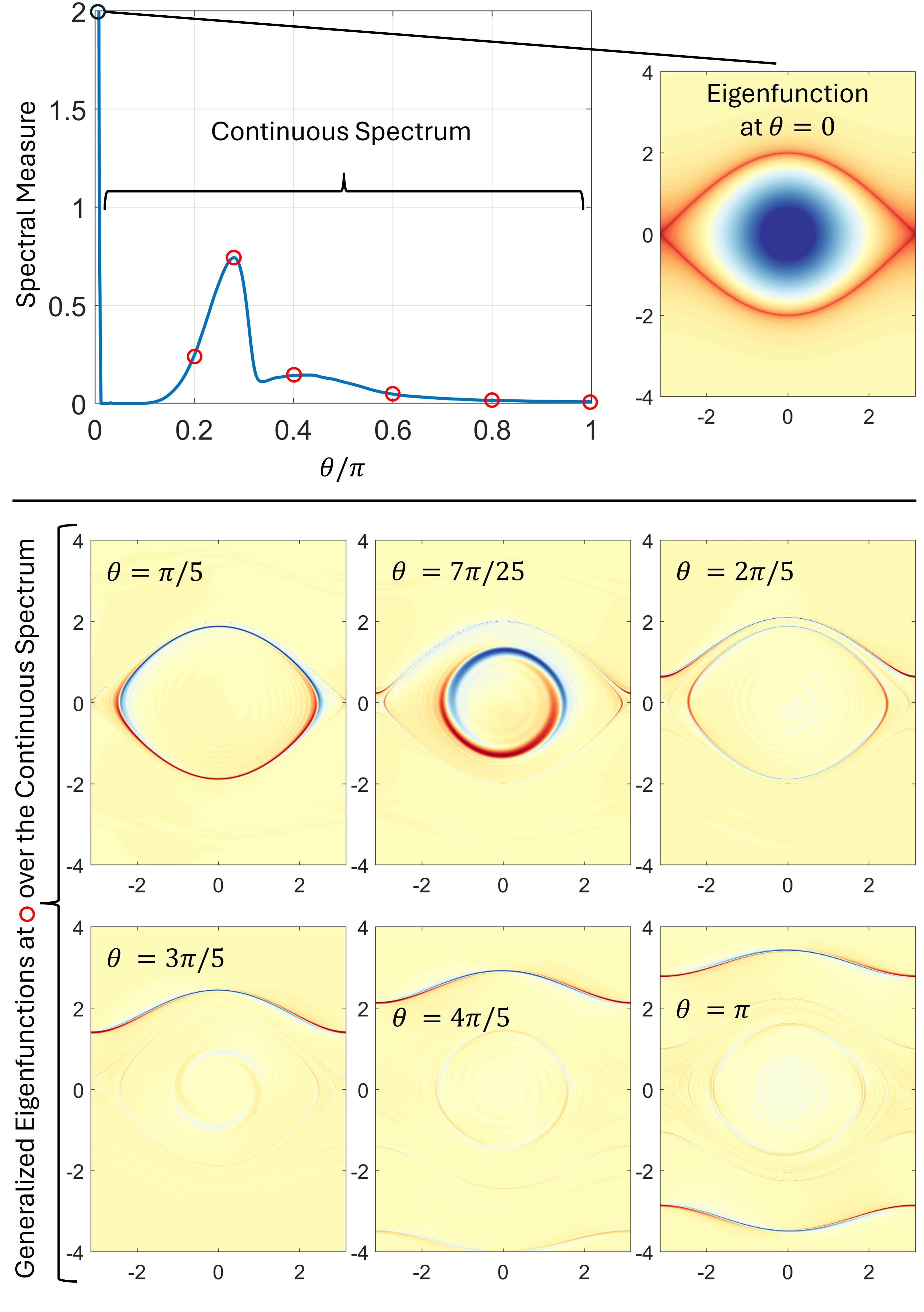}
\caption{The smoothed generalized eigenfunctions of the nonlinear pendulum computed using \cref{alg:RiggedDMD}.\label{fig:pendulum1}}
\end{figure}

\subsection{The Lorenz system}
\label{sec:delay_embedd_example}
The Lorenz (63) system with classical parameter values \cite{lorenz1963deterministic} is the following three coupled ODEs:
$$
\dot{x}=10\left(y-x\right),\quad\dot{y}=x\left(28-z\right)-y,\quad \dot{z}=xy-8z/3.
$$
We consider the dynamics of $\xv=(x,y,z)$ on the Lorenz attractor $\Omega$.
This system has a unique SRB measure\footnote{For a survey of these measures and their definitions, see \cite{young2002srb}.} $\omega$ on $\Omega$ \cite{Tucker2002}, meaning that for Lebesgue-almost every initial condition $\xv_0$ in the basin of attraction of $\Omega$ and compactly supported function $g:\mathbb{R}^3\rightarrow\mathbb{C}$,
\begin{equation}
\label{eq:SRB_conv}
\lim_{M\rightarrow\infty}\frac{1}{M}\sum_{m=0}^{M-1}[\K^{m}g](\xv_0)=\int_{\Omega}g(\xv)\dd \omega(\xv).
\end{equation}
The system is chaotic and strong mixing \cite{luzzatto2005lorenz}. It follows that $\lambda=1$ is the only eigenvalue of $\K$, corresponding to a constant eigenfunction, and that this eigenvalue is simple.

We consider a discrete-time dynamical system by sampling with a time-step $\Delta t=0.05$ and collecting $M=10^4$ snapshot pairs over a single trajectory. This quadrature rule is justified by \cref{eq:SRB_conv}. We use the \texttt{ode45} command in MATLAB to collect the data after an initial burn-in time to ensure that the initial point $\xv(0)$ is (approximately) on the Lorenz attractor. The system is chaotic, so we cannot hope to accurately numerically integrate for long periods. However, convergence is still obtained in the large data limit $M\rightarrow\infty$ due to an effect known as shadowing. As our dictionary, we use time delays of the function
$$
\psi_1(\xv)=g(\xv)=\tanh((xy-5z)/10)-c, \quad \psi_j=\mathcal{K}^{j-1}g,\quad j=2,\ldots,N,
$$
where $c$ is a constant chosen so that $g$ is orthogonal to the one-dimensional space spanned by constant functions. The value of $c\approx -0.5382$ is computed by taking the average of $\tanh((xy-5z)/10)$ over our trajectory. Shifting by $c$ ensures that the spectral measure $\xi_g$ is purely continuous. \cref{fig:lorenz1} shows the spectral measure computed using $N=1000$, the $6$th order kernel in \cref{fig:kernels} and a smoothing parameter $\epsilon=0.05$. There are four noticeable peaks at $\theta\approx0.414,0.497,0.828,$ and $0.994$. Since the spectral measure is continuous, these do not correspond to eigenvalues. Instead, we can compute generalized eigenfunctions corresponding to these values to obtain approximately coherent structures. These are also shown in \cref{fig:lorenz1}. The function labeled (a) is similar to the local spectral projections in \cite[Figure 13]{korda2020data} (see also \cite[Figure 4]{colbrook2023multiverse}), which the authors attributed to an almost-periodic motion of the $z$ component during the time that the state resides in either of the two lobes of the Lorenz attractor. The function labeled (c) corresponds to double the angle $\theta$ of that of (a) and roughly the square of the corresponding function. The structures in (b) and (d) are different and include oscillations within each lobe.

\begin{figure}
\centering
\includegraphics[width=0.45\textwidth]{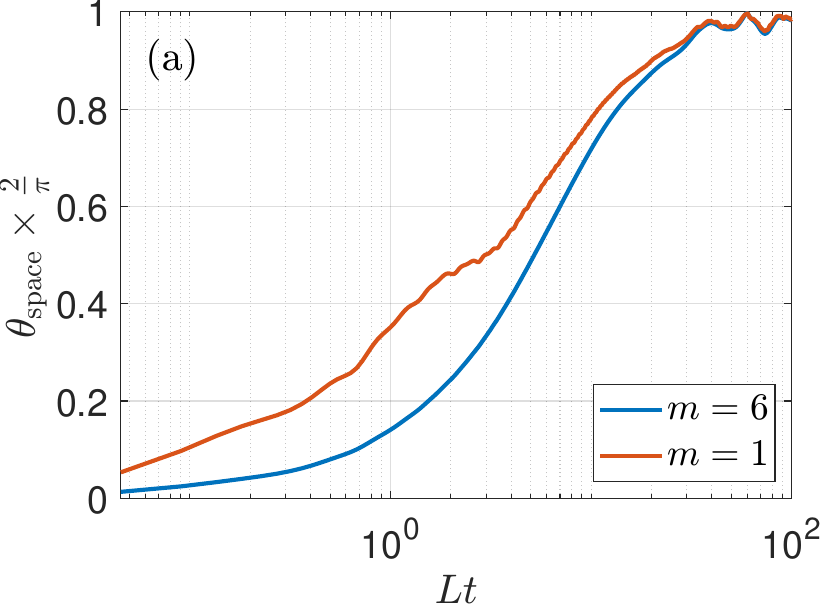}\hfill
\includegraphics[width=0.45\textwidth]{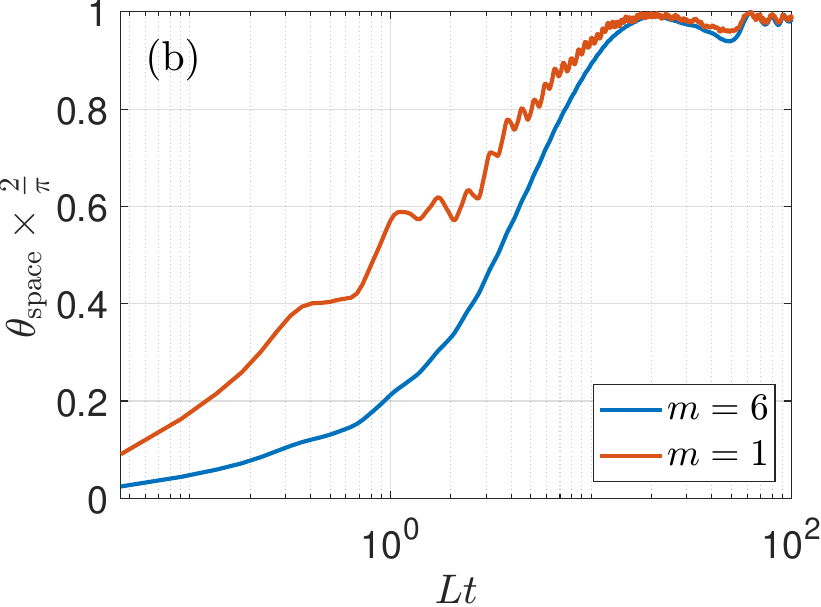}\\
\includegraphics[width=0.45\textwidth]{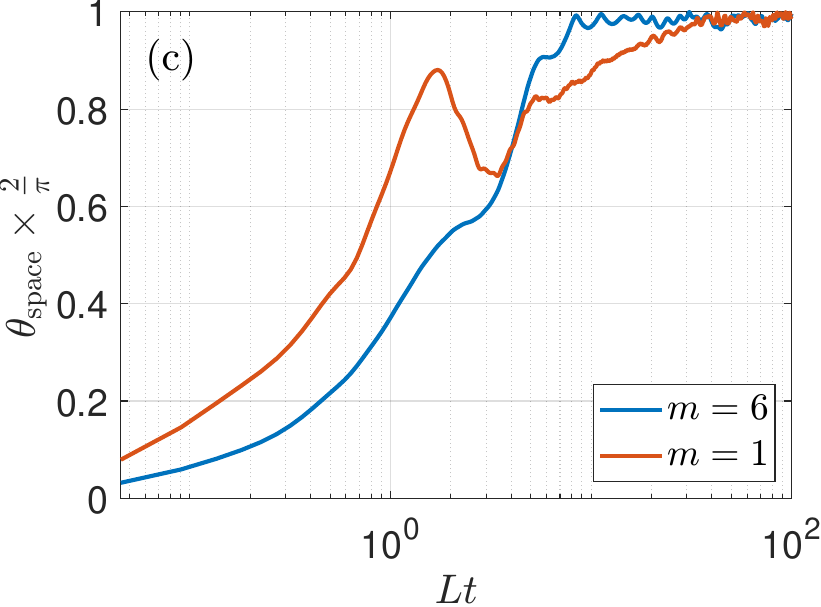}\hfill
\includegraphics[width=0.45\textwidth]{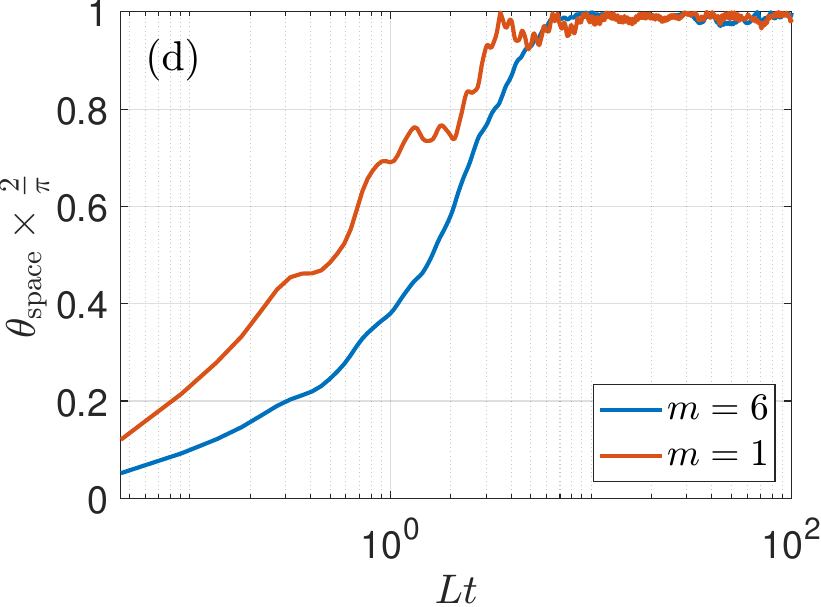}
 \caption{The subspace angle between $g_\theta^\epsilon$ and $\mathcal{K}^ng_\theta^\epsilon$, plotted against $Lt$ with $L$ the maximal Lyapunov exponent of the Lorenz system. The labels (a), (b), (c) and (d) correspond to \cref{fig:lorenz1}. The angle $\theta_{\mathrm{space}}\in[0,\pi/2]$ measures the angle between subspaces, with a value of zero corresponding to perfect coherency and a value of $\pi/2$ corresponding to perfect incoherency.\label{fig:lorenz2}}
\end{figure}

To measure the coherency of each approximate generalized eigenfunction $g_\theta^\epsilon$, we measure the subspace angle between $g_\theta^\epsilon$ and $\mathcal{K}^ng_\theta^\epsilon$. We do this over a larger data set of $10^5$ trajectory points to ensure that the subspace angle is accurate and does not correspond to fitting our original snapshot data. \cref{fig:lorenz2} shows these angles, where we have plotted against $Lt$ with $L$ the maximal Lyapunov exponent of the Lorenz system. In agreement with \cref{approx_cohernecy} and \cref{prop:evc_approx_atom}, we see that our computed generalized eigenfunctions are highly coherent over several Lyapunov time scales. Moreover, higher-order kernels lead to larger levels of coherency.

\subsection{Lid-driven cavity flow}

As our final example, we consider a 2D lid-driven cavity flow at various Reynolds numbers. The flow is a simple model of an incompressible viscous fluid confined to a rectangular box with a moving lid and is a standard benchmark \cite{tseng2001mixing,chella1985fluid,koseff1984lid}. As the Reynolds number increases, the Koopman spectrum undergoes a series of bifurcations from pure point to mixed to pure continuous spectrum \cite{shen1991hopf,arbabi2017study}. It was also shown in \cite[e.g., Figure 7]{arbabi2017study} that for high Reynolds number, it is challenging to capture the flow with discrete DMD expansions. We use the data from \cite{arbabi2017study}, namely, the vorticity field within the cavity at $Re=13000$, $16000$, $19000$ and $30000$ collected over 10000 snapshots. We have also added 20\% Gaussian random noise (SNR of 14) to demonstrate the robustness of Rigged DMD.

\cref{fig:cavity1} shows the spectral measures of the total kinetic energy, computed using \cref{alg:RiggedDMD} with a 10th order kernel, $\epsilon=0.001$ and a dictionary formed by time delays. At $Re=13000$, we see pure point spectrum built from a base frequency. As the Reynolds number increases, the spectrum becomes mixed and then continuous. \cref{fig:cavity2a,fig:cavity2b,fig:cavity2c,fig:cavity2d} show the generalized Koopman modes computed using \cref{alg:RiggedDMD2} with a 6th order kernel and $\epsilon=0.01$ and a dictionary of $10$, $100$, $200$ and $500$ POD modes for $Re=13000$, $16000$, $19000$ and $30000$, respectively. As our values of $\theta$, we pick the Strouhal numbers corresponding to the four peaks visible in the spectral measure at $Re=13000$. We have shown the results with the noisy data and the noise-free data. As expected, larger frequencies are more affected by noise. Nevertheless, we are able to compute these modes over the continuous spectrum in the presence of large noise. There are three reasons for robustness of Rigged DMD. The first is that mpEDMD is equivalent to a constrained total least squares problem (see the discussion in \cite[Section 6.2]{colbrook2023mpedmd}). The second is the smoothing parameter (high-order kernels help, see the discussion in \cite[Section 5.1.2]{colbrook2021rigorous}). The third is the averaging procedure in \cref{alg:RiggedDMD2}.

\begin{figure}
\centering
\includegraphics[width=0.45\textwidth]{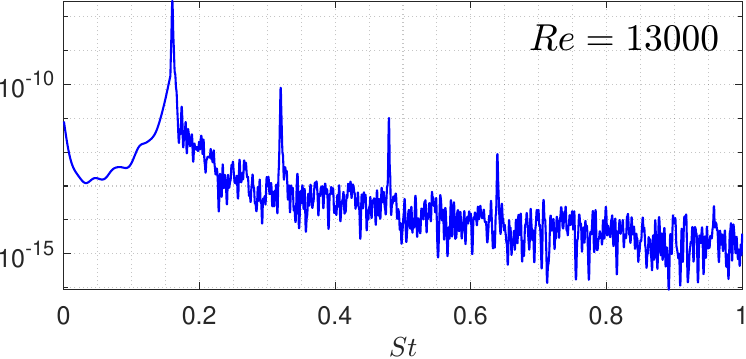}\hfill
\includegraphics[width=0.45\textwidth]{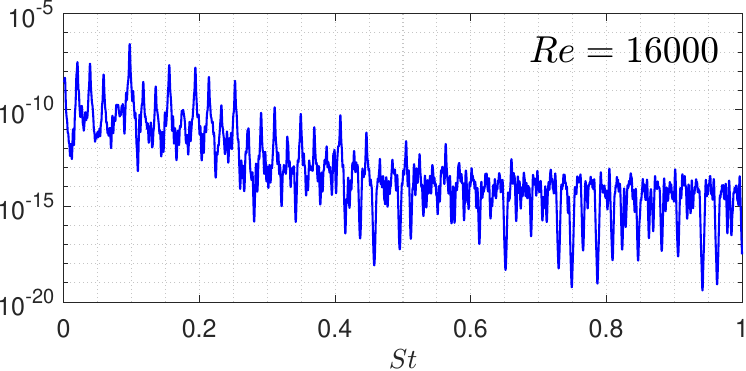}\\
\includegraphics[width=0.45\textwidth]{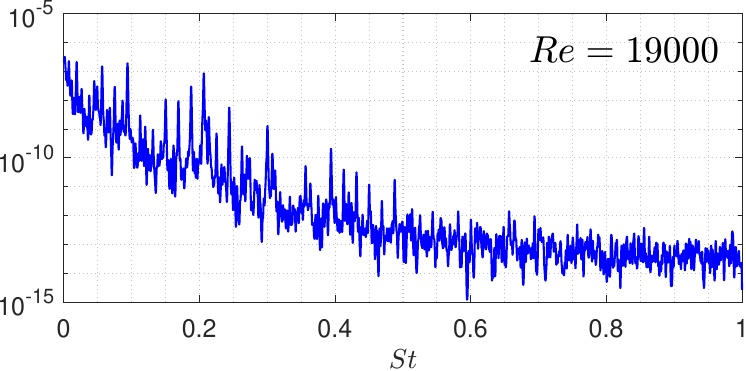}\hfill
\includegraphics[width=0.45\textwidth]{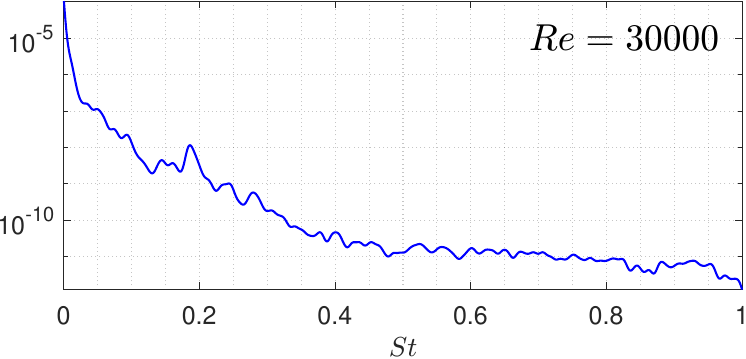}\\
 \caption{Spectral measures of the kinetic energy of the lid-driven cavity flow.\label{fig:cavity1}}
\end{figure}

\begin{figure}
\centering
\includegraphics[width=0.8\textwidth]{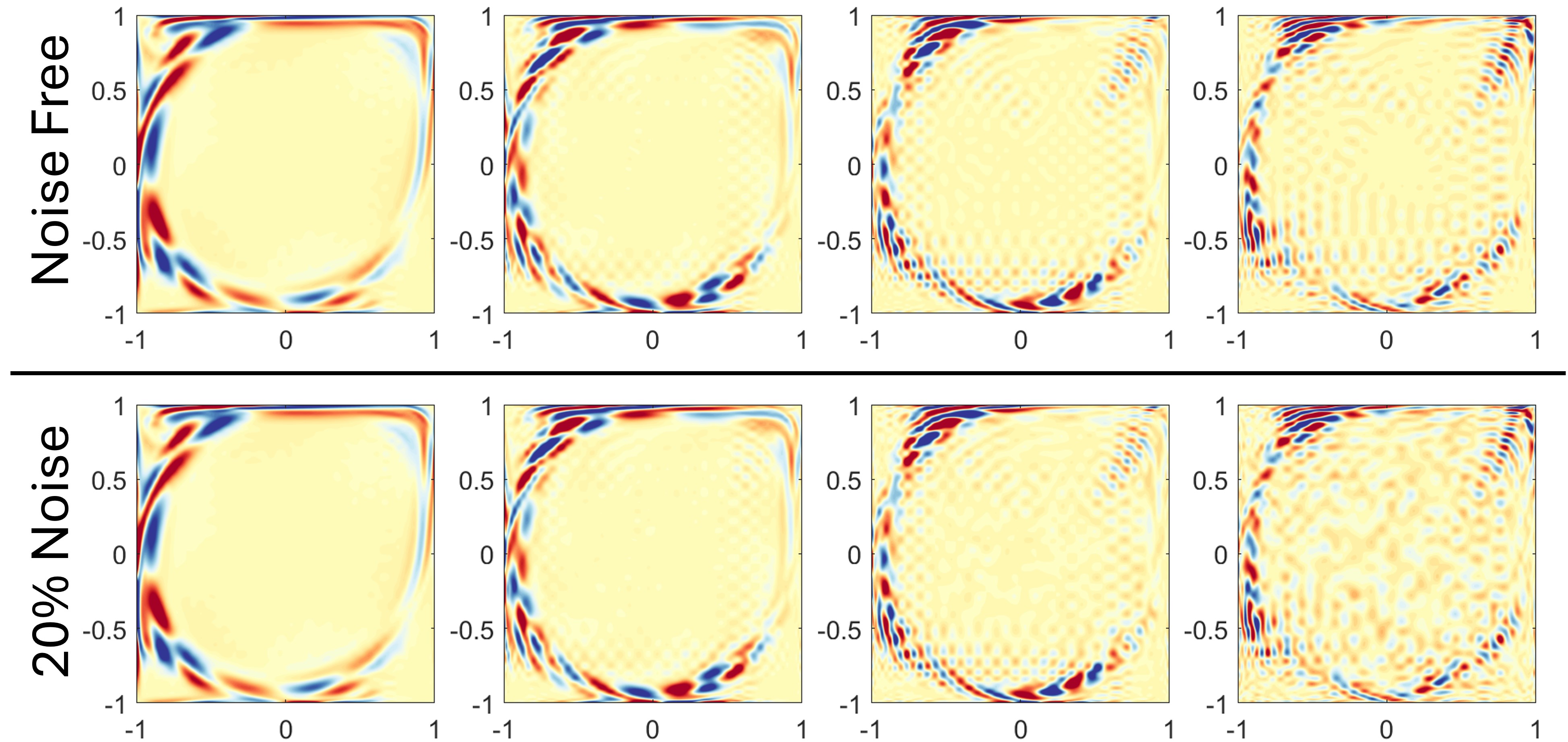}
\caption{Generalized Koopman modes for the cavity flow at $Re=13000$.\label{fig:cavity2a}}
\end{figure}

\begin{figure}
\centering
\includegraphics[width=0.8\textwidth]{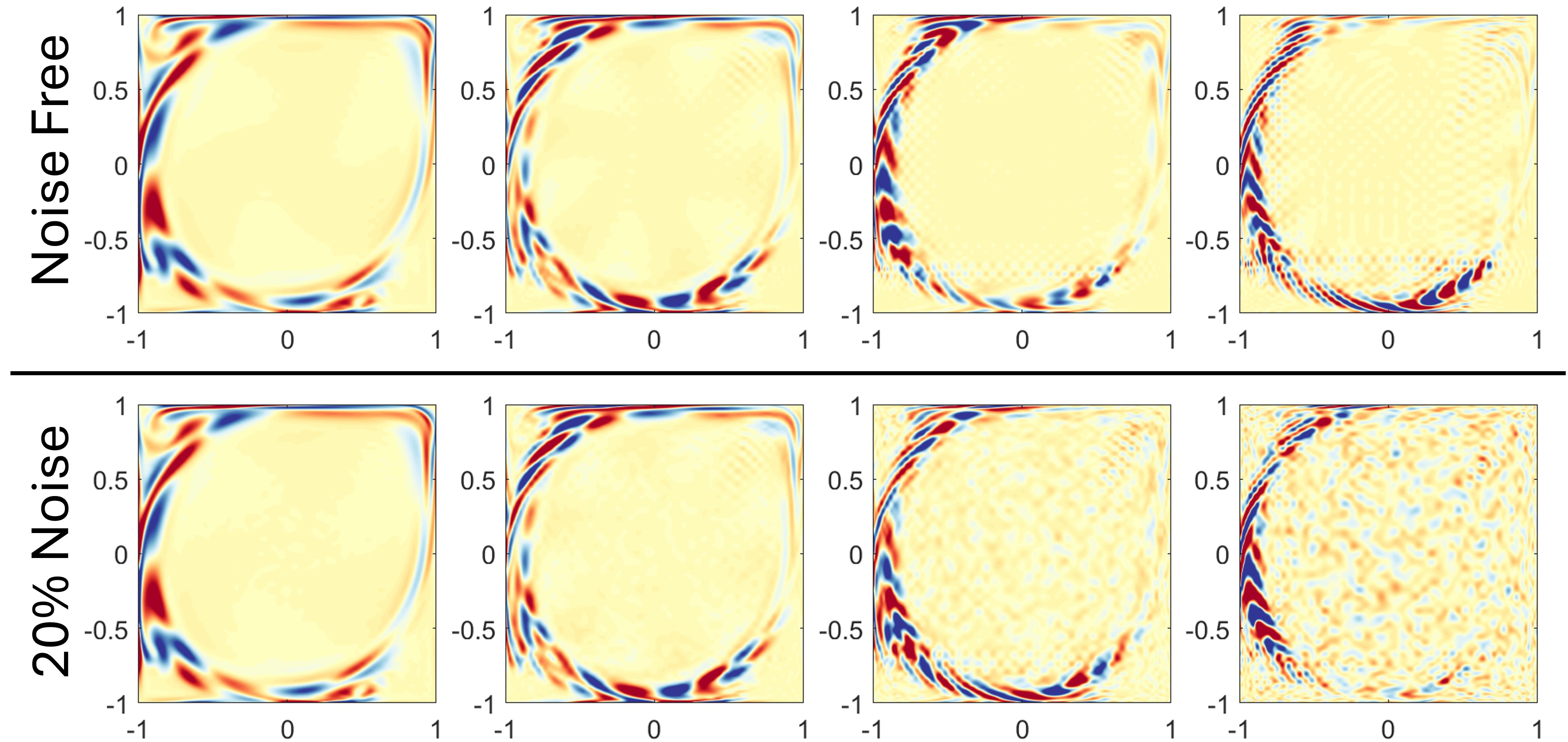}
\caption{Generalized Koopman modes for the cavity flow at $Re=16000$.\label{fig:cavity2b}}
\end{figure}

\begin{figure}
\centering
\includegraphics[width=0.8\textwidth]{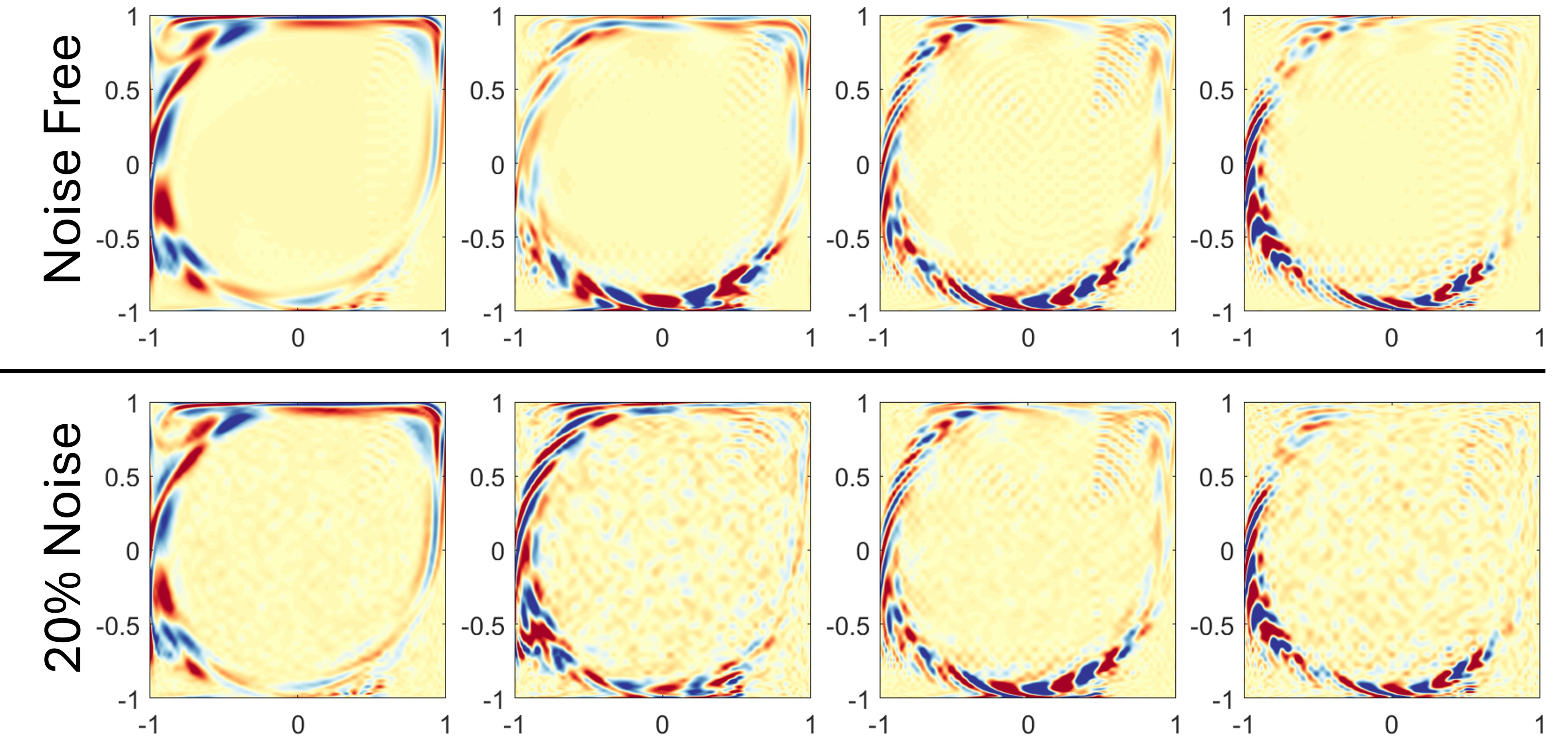}
\caption{Generalized Koopman modes for the cavity flow at $Re=19000$.\label{fig:cavity2c}}
\end{figure}

\begin{figure}
\centering
\includegraphics[width=0.8\textwidth]{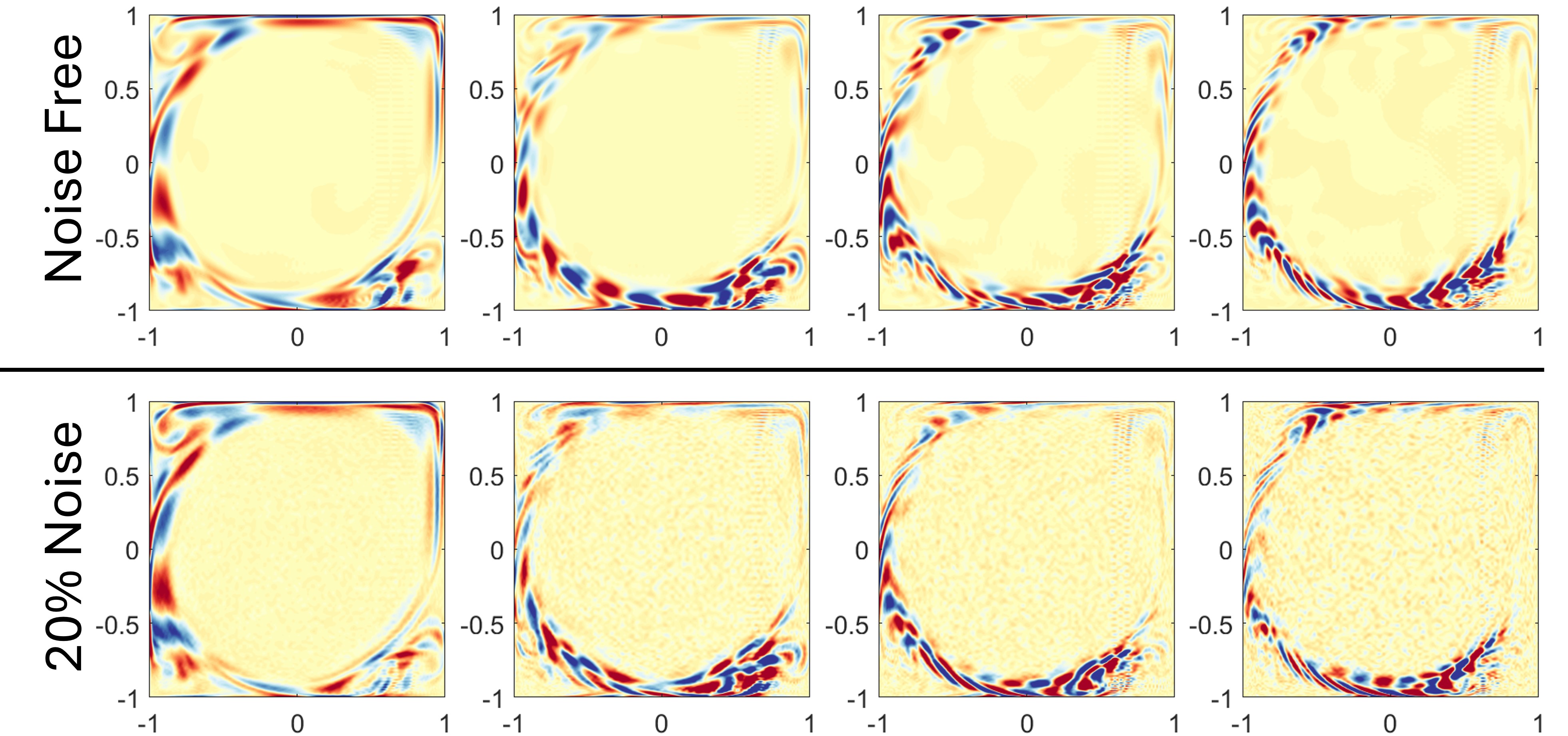}
\caption{Generalized Koopman modes for the cavity flow at $Re=30000$.\label{fig:cavity2d}}
\end{figure}

\section{Conclusions}\label{sec_conclusions}

We have developed an algorithm that computes generalized eigenfunction expansions (mode decompositions) of Koopman operators, even in the presence of continuous spectra. Named Rigged DMD, the algorithm leverages Gelfand's theorem, which asserts that a complete set of generalized eigenfunctions exists within a rigged Hilbert space. Rigged DMD employs unitary approximations of Koopman operators, computed using mpEDMD, to determine the Koopman resolvent. This resolvent is then used to compute smoothed approximations of generalized eigenfunctions for any spectral parameter. We have not only shown examples of rigged Hilbert spaces but also demonstrated how to realize this assumption for general unitary Koopman operators using appropriately weighted vectors derived from delay embedding. In addition to proving convergence theorems for generalized eigenfunctions and spectral measures, we have demonstrated the algorithm's applicability and robustness through several examples.

There is considerable potential for future work, and we briefly discuss two avenues. The first involves applying Rigged DMD to reduced order models and control problems involving continuous spectra. Specifically, Rigged DMD opens the door to control over the interactions between modes corresponding to the continuous and discrete spectra. Second, we relied on the Koopman operator being unitary (i.e., the system is measure-preserving and invertible). We aim to extend Rigged DMD to non-unitary and non-normal Koopman operators, a development that would require adapting these techniques to scenarios where the nuclear spectral theorem does not apply.

Finally, several methods now exist for computing spectral measures of Koopman operators. It would be interesting to explore which of these methods can be extended to compute generalized eigenfunctions. Furthermore, we believe the community is at a stage where comparing these approaches for spectral measures would be beneficial, helping to establish a guideline for identifying the most suitable methods for different types of problems.

\appendix

\section{Strong convergence of the mpEDMD resolvent}\label{sec:mpEDMD_app}
To prove \cref{cor_mpEDMD_res_conv}, we use the following lemma from \cite{colbrook2023mpedmd}. We stress that the second part of the lemma need not hold for methods such as EDMD. In particular, convergence in the strong operator topology (the convergence of EDMD) need not imply convergence of adjoints in the strong operator topology, even if the desired limit is unitary. The polar decomposition in mpEDMD is crucial for strong convergence of the adjoint of the Koopman operator.

\begin{lemma}
\label{lemma_SOT_conv}
Suppose that $\lim_{N\rightarrow\infty}\!\mathrm{dist}(h,V_{N})=0$ for all $h\in L^2(\Omega,\omega)$ and~\cref{quad_convergence} holds.  Then for any $g\in L^2(\Omega,\omega)$ and $\gv_N\in\mathbb{C}^N$ with $\lim_{N\rightarrow\infty}\|g-\Psiv \gv_N\|=0$,
\begin{align}
\lim_{N\rightarrow\infty}\limsup_{M\rightarrow\infty} \|\mathcal{K}^lg-\Psiv\Kv^l\gv_N\|&=0\quad \forall l\in\mathbb{N}\cup\{0\},\label{limit1}\\
\lim_{N\rightarrow\infty}\limsup_{M\rightarrow\infty} \|(\mathcal{K}^*)^lg-\Psiv\Kv^{-l}\gv_N\|&=0\quad \forall l\in\mathbb{N}\cup\{0\}.\label{limit2}
\end{align}
\end{lemma}

\begin{proof}[Proof of \cref{cor_mpEDMD_res_conv}]
Suppose first that $|z|>1$. Since $\|\mathcal{K}\|=1$ and $\|\Gv^{1/2}\Kv\Gv^{-1/2}\|=1$,
\begin{align*}
(\mathcal{K}-zI)^{-1}g&=\frac{-1}{z}(I-\mathcal{K}/z)^{-1}g=\frac{-1}{z}\sum_{l=0}^\infty \frac{1}{z^l}\mathcal{K}^lg\\
(\Kv-z\Iv)^{-1}g&=\frac{-1}{z}(\Iv-\Kv/z)^{-1}\gv_N=\frac{-1}{z}\sum_{l=0}^\infty \frac{1}{z^l}\Kv^l\gv_N.
\end{align*}
The tails of these series can be uniformly bounded, and for any $l\in\mathbb{N}\cup\{0\}$, \cref{limit1} implies that
$$
\lim_{N\rightarrow\infty}\limsup_{M\rightarrow\infty}\left\|\frac{1}{z^l}\mathcal{K}^lg-\Psiv\frac{1}{z^l}\Kv^l\gv_N\right\|=0.
$$
The convergence in \cref{mpEDMD_res_convergence} now follows. Similarly, if $|z|<1$, we may write
\begin{align*}
(\mathcal{K}-zI)^{-1}g&=(I-z\mathcal{K}^*)^{-1}\mathcal{K}^*g=\sum_{l=0}^\infty z^l(\mathcal{K}^*)^{l+1}g\\
(\Kv-z\Iv)^{-1}g&=(\Iv-z\Kv^{-1})^{-1}\Kv^{-1}\gv_N=\sum_{l=0}^\infty {z^l}\Kv^{-l-1}\gv_N.
\end{align*}
We argue similarly using \cref{limit2} to see that \cref{mpEDMD_res_convergence} holds.
\end{proof}

\section{Construction of high-order periodic kernels}
\label{sec:periodic_kernels_app}

\begin{proof}[Proof of \cref{prop:per_summation}]
The decay condition \cref{eq:decay_bound} in \cref{def:mth_order_kernel} ensures that the series defining $K_\epsilon^{\mathbb{T}}$ is absolutely convergent and hence that $K_\epsilon^{\mathbb{T}}$ is continuous. Similarly,
$$
\int_{-\pi}^\pi K_{\epsilon}^\mathbb{T}(\theta)\dd\theta=\int_{\mathbb{R}}K_\epsilon(x)\dd x=1.
$$
We may write
$$
\int_{-\pi}^\pi K_{\epsilon}^\mathbb{T}(-\theta)e^{in\theta}\dd\theta-1=\int_{\mathbb{R}}K_\epsilon(x)\left[e^{-inx}-1\right]\dd x=[K_\epsilon\ast\rho](0),
$$
where $\rho(x)=e^{inx}-1$. Since $\rho$ is a bounded, smooth function, and $\rho(0)=0$, we may apply \cite[Theorem 5.2]{colbrook2021computing} to deduce that \cref{eq:fourier_cond} holds.

We are left with proving the concentration bound \cref{def:unitary_decay}. By directly plugging \cref{eq:decay_bound} into the definition of $K_{\epsilon}^\mathbb{T}$, we see that
\begin{equation}
\label{eq:periodic_splo2}
|K_{\epsilon}^\mathbb{T}(\theta)|\lesssim \frac{1}{\epsilon}\sum_{n\in\mathbb{Z}} \frac{1}{(1+\left|\theta+2\pi n\right|/\epsilon)^{m+1}}= \sum_{n\in\mathbb{Z}}\frac{\epsilon^m}{(\epsilon+|\theta+2\pi n|)^{m+1}}.
\end{equation}
For any non-zero integer $n$, we have $|\theta+2\pi n|\geq \pi|n|$ and hence
$$
\sum_{n\in\mathbb{Z}\backslash{\{0\}}}\frac{1}{(\epsilon+|\theta+2\pi n|)^{m+1}}\leq
\sum_{n\in\mathbb{Z}\backslash{\{0\}}}\frac{1}{|\theta+2\pi n|^{m+1}}<\infty.
$$
It follows that
$$
|K_{\epsilon}^\mathbb{T}(\theta)|\lesssim \epsilon^m+\frac{\epsilon^m}{(\epsilon+|\theta|)^{m+1}},
$$
where the second term on the right-hand side comes from the $n=0$ term in the series in \cref{eq:periodic_splo2}. The decay condition \cref{def:unitary_decay} now follows.
\end{proof}

\begin{proof}[Proof of \cref{prop:carath_form}]
The corresponding periodic kernel is given by
$$
K_\epsilon^{\mathbb{T}}(\theta)=\frac{1}{2\pi i} \sum_{j=1}^{m}\sum_{n\in\mathbb{Z}}\left(\frac{\alpha_j}{\theta+2\pi n-\epsilon a_j}-\frac{\beta_j}{\theta+2\pi n-\epsilon b_j}\right),
$$
where the rearrangement of the sum is justified by its absolute convergence. Since
$$
\sum_{n\in\mathbb{Z}} \frac{1}{z+2\pi n}=\frac{1}{2}\cot\left(\frac{z}{2}\right)=\frac{-i}{2}\frac{1+e^{iz}}{1-e^{iz}},
$$
the kernel can be re-written as
\begin{align*}
K_\epsilon^{\mathbb{T}}(\theta)&\!=\!\frac{1}{2\pi i}\!\sum_{j=1}^{m}\left(\!\alpha_j\frac{1}{2}\cot\left(\frac{\theta-\epsilon a_j}{2}\right)\!-\!\beta_j\frac{1}{2}\cot\left(\frac{\theta-\epsilon b_j}{2}\right) \!\right)\notag\\
&\!=\!\frac{-1}{4\pi}\!\sum_{j=1}^{m}\!\left(\!\alpha_j\frac{1+e^{i\theta-i\epsilon a_j}}{1-e^{i\theta-i\epsilon a_j}}\!-\!\beta_j\frac{1+e^{i\theta-i\epsilon b_j}}{1-e^{i\theta-i\epsilon b_j}}\!\right)\!=\!\frac{-1}{4\pi}\!\sum_{j=1}^{m}\!\left(\!\alpha_j\frac{e^{-i\theta}+e^{-i\epsilon a_j}}{e^{-i\theta}-e^{-i\epsilon a_j}}\!-\!\beta_j\frac{e^{-i\theta}+e^{-i\epsilon b_j}}{e^{-i\theta}-e^{-i\epsilon b_j}}\!\right)\!\!.
\end{align*}
The formula in \cref{carathform1} follows. If $b_j = \overline{a_j}$ and $\beta_j=\overline{\alpha_j}$, since $\langle F_{\mathcal{E}}(z) g,g\rangle = -\overline{\langle F_{\mathcal{E}}(1/\overline{z}) g,g\rangle}$,
\begin{align*}
\left[K_\epsilon^{\mathbb{T}}\ast\xi_g\right](\theta_0)&=\frac{-1}{2\pi}\sum_{j=1}^{m}{\rm Re}\left(\alpha_j\left\langle F_\mathcal{E}(e^{i\theta-i\epsilon a_j})g,g\right\rangle\right)\\
&=\frac{-1}{2\pi}\sum_{j=1}^{m}{\rm Re}\left(\alpha_j \left\langle (\K-e^{i\theta-i\epsilon a_j}I)^{-1}g,(\K+e^{i\theta-i\epsilon a_j}I)^*g\right\rangle  \right).
\end{align*}
\end{proof}

\section{Convergence of wave-packet approximations}\label{sec:conv_WPA_app}
We use the following notation throughout this section. Given $m\in\mathbb{N}$, we let $\{K_{\epsilon}^\mathbb{T}\}$ be an $m$th order kernel for $\pp$. Let $\mathcal{K}:\mathcal{H}\rightarrow\mathcal{H}$ be a unitary operator on a rigged Hilbert space $\mathcal{S}\subset \mathcal{H}\subset\mathcal{S}^*$ with $\mathcal{K}\mathcal{S}\subset\mathcal{S}$. Denote the absolutely continuous component of ${\rm Sp}(\mathcal{K})$ by ${\rm Sp}_{\rm ac}(\mathcal{K})$ and, given $g,\phi\in\mathcal{S}$, let $\rho_{g,\phi}$ be the Radon--Nikodym derivative of the absolutely continuous component of the spectral measure $\xi_{g,\phi}=\langle\mathcal{E}g,\phi\rangle$, with $\xi_g = \xi_{g,g}$ for brevity.

Before proving the main results, we note that the smoothed Koopman mode expansion in~\cref{eq:smoothedKMD} can be simplified with, essentially, a ``change-of-basis" for the generalized eigenspace:
\begin{equation}\label{eqn:simple_wave_packet}
g_\theta^\epsilon=\int_\pp K_\epsilon(\theta-\varphi)\, \langle u_{\varphi}^*|g\rangle u_{\varphi}\dd\xi_{g}(\varphi) \qquad \text{for every}\qquad g\in\mathcal{S},
\end{equation}
where $u_\theta$ depends on $g$ and satisfies
$\langle \mathcal{K}u_\theta | \phi \rangle = \lambda \langle \psi | \phi \rangle$ for all $\phi \in \mathcal{S}$
with $\lambda=\exp(i\theta)$. The action of $g_\theta^\epsilon$ on a function $f\in\mathcal{S}$ is given by
\begin{equation}\label{eqn:simple_wave_packet_action}
\langle g_\theta^\epsilon | f \rangle =\int_\pp  K_\epsilon(\theta-\varphi)\, \langle u_{\varphi}^*|g\rangle \langle u_{\varphi}|f\rangle \dd\xi_{g}(\varphi)\qquad \text{for every}\qquad g\in\mathcal{S}.
\end{equation}
The corresponding expansion of the spectral measure in the rigged Hilbert space is
\begin{equation}\label{eq:simple_diag_pvm}
[\mathcal{E}(S)]g = \int_S  \langle u_{\varphi}^*|g\rangle u_{\varphi} \dd\xi_{g}(\varphi)\qquad \forall g\in\mathcal{S}.
\end{equation}
After taking an inner product with $g$, \cref{eq:simple_diag_pvm} implies that
$$
\xi_{g}(S)=
\langle [\mathcal{E}(S)]g,g \rangle=
\langle [\mathcal{E}(S)]g|\overline{g} \rangle=
\int_S  \langle u_{\varphi}^*|g\rangle \langle u_{\varphi}|\overline{g}\rangle \dd\xi_{g}(\varphi)=
\int_S  |\langle u_{\varphi}^*|g\rangle|^2 \dd\xi_{g}(\varphi).
$$
It follows that $|\langle u_{\varphi}^*|g\rangle|^2=1$ $\xi_{g}$-almost everywhere.\footnote{This is a convenient consequence of using the spectral measure $\xi_g$ to normalize the generalized eigenfunctions.} We can, without loss of generality, scale each generalized eigenfunction by a phase factor so that $\langle u_{\varphi}^*|g\rangle=1$. As a consequence of this convention, upon taking an inner product with $f\in\mathcal{S}$, we see that $\smash{\dd\xi_{g,f}(\varphi) = \langle u_{\varphi}|\overline{f}\rangle \dd\xi_g(\varphi)}$.

\begin{proof}[Proof of~\cref{thm:conv_rates}]
Since the spectrum of $\mathcal{K}$ is absolutely continuous on the interval $\mathcal{I}$, the arguments above show that $
\rho_{g,\phi}(\varphi)=\rho_{g}(\varphi)\langle u_{\varphi}|\overline{\phi}\rangle$ for $\varphi\in\mathcal{I}$.
\cref{thm:conv_rates} now follows by applying~\cref{thm:unitary_pointwise_convergence} to the complex Borel measure $\xi_{g,{\phi}}$ after substitution in~\cref{eqn:simple_wave_packet_action}.
\end{proof}

\begin{proof}[Proof of~\cref{thm:conv_rates2}]
Under the hypotheses of the theorem, the integrand of~\cref{eqn:simple_wave_packet} is equivalent to convolution with a complex Borel measure (with finite total variation) on $\pp$ for each point $\xv\in\mathcal{A}$. The uniform integrability and continuity assumptions ensure that the total variation and H\"older norms are bounded independently of $\xv\in\mathcal{A}$. \cref{thm:conv_rates2} now follows from~\cref{thm:unitary_pointwise_convergence}.
\end{proof}

We now prove the convergence of the approximate eigenvector sequence.
\begin{proof}[Proof of~\cref{prop:evc_approx_atom}]
Using the functional calculus, we have.
$$
(\K-e^{i\theta_0}I)g_{\theta_0}^\epsilon=\left[\int_{\pp} (e^{i\varphi}-e^{i\theta_0})K_{\epsilon}^\mathbb{T}(\theta_0-\varphi)\dd \mathcal{E}(\varphi)\right]g.
$$
This implies that
$$
\|(\K-e^{i\theta_0}I)g_{\theta_0}^\epsilon\|^2=\int_{\pp} |e^{i\varphi}-e^{i\theta_0}|^2|K_{\epsilon}^\mathbb{T}(\theta_0-\varphi)|^2\dd \xi_g(\varphi).
$$
Similarly, we have
$$
\|g_{\theta_0}^\epsilon\|^2=\int_{\pp} |K_{\epsilon}^\mathbb{T}(\theta_0-\varphi)|^2\dd \xi_g(\varphi).
$$
Suppose that $\epsilon\eta<\delta$ and let
$$
I_\delta=\int_{|\theta_0-\varphi|\leq \delta} |K_{\epsilon}^\mathbb{T}(\theta_0-\varphi)|^2\dd \xi_g(\varphi)\geq 
\left(\inf_{|\theta|<\eta\epsilon}{\epsilon|K_\epsilon^\mathbb{T}(\theta)|}\right)^2\frac{1}{\epsilon^2}\int_{|\theta_0-\varphi|\leq \eta\epsilon} 1\dd \xi_g(\varphi).
$$
The two bounds in \cref{prop:evc_approx_atom} imply that $\liminf_{\epsilon\downarrow 0}I_\delta>0$. We have
\begin{align*}
&\int_{\pp} |e^{i\varphi}-e^{i\theta_0}|^2|K_{\epsilon}^\mathbb{T}(\theta_0-\varphi)|^2\dd \xi_g(\varphi)\\
&\quad\quad\quad\leq \left(\sup_{|\theta|\leq \delta}|1-e^{i\theta}|^2\right)I_\delta+\int_{|\theta_0-\varphi|> \delta}|e^{i\varphi}-e^{i\theta_0}|^2|K_{\epsilon}^\mathbb{T}(\theta_0-\varphi)|^2\dd \xi_g(\varphi)\\
&\quad\quad\quad\lesssim
\delta^2I_\delta+\int_{|\theta_0-\varphi|> \delta}|e^{i\varphi}-e^{i\theta_0}|^2|K_{\epsilon}^\mathbb{T}(\theta_0-\varphi)|^2\dd \xi_g(\varphi).
\end{align*}
Using the concentration condition in \cref{def:unitary_decay}, we see that
$$
\lim_{\epsilon\downarrow 0} \int_{|\theta_0-\varphi|> \delta}|e^{i\varphi}-e^{i\theta_0}|^2|K_{\epsilon}^\mathbb{T}(\theta_0-\varphi)|^2\dd \xi_g(\varphi)=0.
$$
Similarly, we have
$$
\int_{\pp} |K_{\epsilon}^\mathbb{T}(\theta_0-\varphi)|^2\dd \xi_g(\varphi)\geq I_\delta-
\int_{|\theta_0-\varphi|> \delta}|K_{\epsilon}^\mathbb{T}(\theta_0-\varphi)|^2\dd \xi_g(\varphi)
$$
and 
$$
\lim_{\epsilon\downarrow 0} \int_{|\theta_0-\varphi|> \delta}|K_{\epsilon}^\mathbb{T}(\theta_0-\varphi)|^2\dd \xi_g(\varphi)=0.
$$
Hence,
$$
\limsup_{\epsilon\downarrow 0}\frac{\|(\K-e^{i\theta_0}I)g_{\theta_0}^\epsilon\|^2}{\|g_{\theta_0}^\epsilon\|^2}
\lesssim\limsup_{\epsilon\downarrow 0}\frac{\delta^2I_\delta+\int_{|\theta_0-\varphi|> \delta}|e^{i\varphi}-e^{i\theta_0}|^2|K_{\epsilon}^\mathbb{T}(\theta_0-\varphi)|^2\dd \xi_g(\varphi)}{I_\delta-
\int_{|\theta_0-\varphi|> \delta}|K_{\epsilon}^\mathbb{T}(\theta_0-\varphi)|^2\dd \xi_g(\varphi)}= \delta^2.
$$
Since $\delta>0$ was arbitrary, the result now follows.
\end{proof}

\begin{proof}[Proof of \cref{prop:evc_atom}]
Taking $\mathcal{P}_{e^{i\theta_0}}=0$ in the case that $e^{i\theta_0}$ is not an eigenvalue,
$$
\frac{1}{K_\epsilon^\mathbb{T}(0)}[K_{\epsilon}^\mathbb{T}\ast \mathcal{E}](\theta_0)\mathcal{P}_{e^{i\theta_0}}g=
\frac{1}{K_\epsilon^\mathbb{T}(0)}\int_{\pp}K_{\epsilon}^\mathbb{T}(\theta_0-\varphi)\dd \mathcal{E}(\varphi) \mathcal{P}_{e^{i\theta_0}}g=\mathcal{P}_{e^{i\theta_0}}g.
$$
It follows that we can define
$$
g_\epsilon'=\frac{1}{K_\epsilon^\mathbb{T}(0)}[K_{\epsilon}^\mathbb{T}\ast \mathcal{E}](\theta_0)g-\mathcal{P}_{e^{i\theta_0}}g=\frac{1}{K_\epsilon^\mathbb{T}(0)}[K_{\epsilon}^\mathbb{T}\ast \mathcal{E}](\theta_0)(g-\mathcal{P}_{e^{i\theta_0}}g).
$$
From the functional calculus, we have
\begin{equation}
\label{atom_recov}
\|g_\epsilon'\|^2=\int_{\pp} \frac{|K_{\epsilon}^\mathbb{T}(\theta_0-\varphi)|^2}{|K_\epsilon^\mathbb{T}(0)|^2}\dd \xi_{g-\mathcal{P}_{e^{i\theta_0}}g}(\varphi).
\end{equation}
The concentration bound in \cref{def:unitary_decay} and the condition
$
\liminf_{\epsilon\downarrow 0}{\epsilon|K_\epsilon^\mathbb{T}(0)|}>0.
$
imply that
$$
\frac{|K_{\epsilon}^\mathbb{T}(\theta_0-\varphi)|^2}{|K_\epsilon^\mathbb{T}(0)|^2}\lesssim \frac{1}{\epsilon^2|K_\epsilon^\mathbb{T}(0)|^2}\frac{1}{(1+|\theta_0-\varphi|/\epsilon|)^{2m+2}}.
$$
is uniformly bounded and converges to zero as $\epsilon\downarrow 0$ whenever $\varphi\neq \theta_0$. Since $\xi_{g-\mathcal{P}_{e^{i\theta_0}}g}(\{\theta_0\})=0$, we can apply the dominated convergence theorem to the right-hand side of \eqref{atom_recov} to see that $\lim_{\epsilon\downarrow}\|g_\epsilon'\|^2=0$, which proves the result.
\end{proof}

\section{Construction of a rigged Hilbert spaces for sparse unitary operators}\label{sec:unitary_RHS_app}

\begin{proof}[Proof of \cref{lemma:weightedl2}]
Let $\{e_n: n\in\mathbb{N}\}$ denote the canonical basis for $\ell^2(\mathbb{N})$. Then, for any weight function $\rho$, the vectors $\{e_n/\sqrt{\rho(n)}:n\in\mathbb{N}\}$ form an orthonormal basis of $l^2_{\rho}(\mathbb{N})$. We may represent the inclusion map $\iota$ as
$$
\iota u = \sum_{j=1}^\infty   \sqrt{\frac{\tau(j)}{\tau'(j)}}\left\langle u, \frac{e_j}{\sqrt{\tau'(j)}} \right\rangle_{\tau'}   \frac{e_j}{\sqrt{\tau(j)}}.
$$
It follows that if
\begin{equation}
\label{nuclear_needed1}
\sum_{j=1}^\infty   \sqrt{\frac{\tau(j)}{\tau'(j)}}<\infty,
\end{equation}
then $\iota$ is nuclear. To achieve part (ii), it is enough to consider finite linear combinations of the basis vectors $\{e_n/\sqrt{\tau'(n)}:n\in\mathbb{N}\}$ of $l^2_{\tau'}(\mathbb{N})$. Let
$
u=\sum_{j=1}^N u_j e_j.
$
Then
$$
\|u\|_{\tau'}^2 = \sum_{j=1}^N |u_j|^2\tau'(j),\qquad \|Au\|_{\tau}^2 = \sum_{j=1}^N |(Au)_j|^2\tau(j).
$$
The assumption that $A$ is sparse implies that there exists an increasing function $f:\mathbb{N}\rightarrow\mathbb{N}$ with $f(n)\geq n$ such that $A_{jk}=0$ if $j>f(k)$ and hence we can write
$$
(Au)_j=\sum_{k\in\{1,\ldots,N\},f(k)\geq j} A_{jk}u_k.
$$
Hence by H\"older's inequality,
$$
|(Au)_j|^2\leq  \sum_{k\in\{1,\ldots,N\},f(k)\geq j} |A_{jk}|^2\sum_{k\in\{1,\ldots,N\},f(k)\geq j} |u_k|^2\leq \|A^*\|^2\sum_{k\in\{1,\ldots,N\},f(k)\geq j} |u_k|^2.
$$
Since $A$ is unitary on $l^2(\mathbb{N})$, $\|A^*\|=1$. It follows that
$$
\|Au\|_{\tau}^2\leq  \sum_{j=1}^N\sum_{k\in\{1,\ldots,N\},f(k)\geq j} |u_k|^2\tau(j)=\sum_{k=1}^N|u_k|^2\sum_{j=1}^{f(k)}\tau(j).
$$
It follows that if
\begin{equation}
\label{nuclear_needed2}
\sum_{j=1}^{f(k)}\tau(j)\leq\tau'(k),
\end{equation}
then $\|A u\|_\tau\leq \|u\|_{\tau'}$. To finish the proof, we select $\tau'$ so that both \cref{nuclear_needed1} and \cref{nuclear_needed2} hold.
\end{proof}

\section*{Acknowledgments}
We wish to thank the UK Spectral Theory Network (supported by the Isaac Newton Institute and EPSRC grant EP/V521929/1) and the COST action MAT-DYN-NET (CA18232), for facilitating in-person collaboration between the authors. 
\bibliographystyle{siamplain}
\bibliography{Koopman}

\end{document}